\documentclass[a4paper,fleqn,11pt]{amsart}

\usepackage{times,cite}
\usepackage{graphicx}
\usepackage{amsfonts,amssymb,mathtools,amsthm}
\usepackage{color}
\usepackage{hyperref}
\usepackage{rotating}
\usepackage{mathtools}
\usepackage{tikz}
\usepackage{ulem}
\usepackage{geometry}
\usepackage[foot]{amsaddr}

\numberwithin{equation}{section}

\theoremstyle{plain}
\newtheorem{theorem}{Theorem}[section]
\newtheorem*{theorem*}{Theorem}
\newtheorem{lemma}[theorem]{Lemma} 
 
\newtheorem{corollary}[theorem]{Corollary} 
\theoremstyle{remark}
\newtheorem{remark}[theorem]{Remark}  

\newcommand{\e}{^\varepsilon}
\newcommand{\eps}{{\varepsilon}}
\newcommand{\al}{\alpha}
\newcommand{\be}{\beta}

\renewcommand{\d}{\mathrm{d}}
\newcommand{\cupl}{\bigcup\limits}
\newcommand{\suml}{\sum\limits}

\newcommand{\liml}{\lim\limits}

\newcommand{\dist}{\mathrm{dist}}
\newcommand{\dom}{\mathrm{dom}}

\newcommand{\N}{\mathbb{N}}

\newcommand{\R}{\mathbb{R}}

\renewcommand{\S}{\mathfrak{S}}

\renewcommand{\L}{\mathsf{L}^2} 
\renewcommand{\H}{\mathsf{H}}
\newcommand{\HS}{\mathcal{H}}

\newcommand{\Ak}{\mathbf{A}} 
\newcommand{\ak}{\mathbf{a}} 
\renewcommand{\a}{\mathfrak{a}} 

\renewcommand{\u}{\mathbf{u}}
\newcommand{\vv}{\mathbf{v}}

\newcommand{\ess}{\mathrm{ess}}
\newcommand{\disc}{\mathrm{disc}}
\newcommand{\acc}{\mathfrak{acc}}

\newcommand{\I}{\mathcal{I}}
\newcommand{\Id}{\mathrm{I}}

\newcommand{\dC}{{\mathbb{C}}}
\newcommand{\dR}{{\mathbb{R}}}
\newcommand{\dN}{{\mathbb{N}}}

\newcommand{\cG}{{\mathcal G}}
\newcommand{\cH}{{\mathcal H}}

\DeclareMathOperator{\ran}{ran}
\newcommand{\ds}{\displaystyle}

\begin{document}

\title[Differential operators with prescribed spectrum]{Construction of self-adjoint differential operators with prescribed spectral properties}

\author[J. Behrndt]{Jussi Behrndt$^1$}
\email{behrndt@tugraz.at}
\author[A. Khrabustovskyi]{Andrii Khrabustovskyi$^{1,2}$}
\email{khrabustovskyi@math.tugraz.at}

\address{$^1$ Institute of Applied Mathematics, Graz University of Technology, Steyrergasse
30, 8010 Graz, Austria}    
\address{$^2$ Department of Physics, Faculty of Science, University of Hradec Kr\'alov\'e, Rokitansk\'eho 62,	500 03 Hradec Kr\'alov\'e, Czech Republic}


\maketitle

\begin{center}
\textit{Dedicated to the memory of our friend and colleague Hagen Neidhardt} 
\end{center}

\begin{abstract}
In this expository article some spectral properties of self-adjoint differential operators 
are investigated. The main objective is to illustrate and (partly) review how one can construct domains or potentials such that the 
essential or discrete spectrum of
a Schr\"{o}dinger operator of a certain type (e.g. the Neumann Laplacian) coincides with a predefined subset of the real line.
Another aim is to emphasize that the spectrum of a differential operator on a bounded domain or bounded interval is not necessarily discrete, that is, 
eigenvalues of infinite multiplicity, continuous spectrum, and eigenvalues embedded in the continuous spectrum may be present. This {\it unusual} 
spectral effect is, very roughly speaking, 
caused by (at least) one of the following three reasons: The bounded domain has a rough boundary, the potential is singular, or the boundary condition
is nonstandard. In three separate explicit constructions we demonstrate how each of these 
possibilities leads to a Schr\"{o}dinger operator with prescribed essential spectrum. 
\end{abstract}

\section{Introduction}

This paper is concerned with spectral theory of self-adjoint differential operators in Hilbert spaces. Before we explain in more detail 
the topics and results we briefly familiarize the reader with the notions of discrete spectrum and essential spectrum, that play a key role here. 
Let $A$ be a (typically unbounded) 
self-adjoint operator in an infinite dimensional complex Hilbert space $\HS$, see also the beginning of Section~\ref{sec3} 
for more details on the {\it adjoint} of unbounded operators and the notion {\it self-adjoint}. The {\it spectrum} $\sigma(A)$ of $A$ is a closed subset of the real line 
(which is unbounded if and only if $A$ is unbounded) that consists of all those points $\lambda$ such that $A-\lambda$ does not admit 
a bounded inverse. In the case that $A-\lambda$ is not invertible $\lambda$ is called an {\it eigenvalue} of $A$ and belongs to the {\it point spectrum};
in the case that $(A-\lambda)^{-1}$ exists as an unbounded operator the point $\lambda$ belongs to the {\it continuous spectrum}. An eigenvalue is
{\it discrete} if it is an isolated point in $\sigma(A)$ and the {\it eigenspace} $\ker(A-\lambda)$ is finite dimensional. This subset of the spectrum of $A$ is denoted
by $\sigma_\disc(A)$; the complement of the discrete spectrum in $\sigma(A)$ is called the {\it essential spectrum} of $A$ and the notation $\sigma_\ess(A)$
is used for this set. It is clear that
$$
\sigma(A)=\sigma_\disc(A)\,\dot\cup\,\sigma_\ess(A)
$$
and that $\sigma_\ess(A)$ consists of all those spectral points which are in the continuous spectrum, all 
eigenvalues embedded in the continuous spectrum
and all isolated eigenvalues of infinite multiplicity. For the intuition it may be helpful to keep in mind that essential spectrum can only appear in an infinite
dimensional Hilbert space, whereas the spectrum of any matrix is necessarily discrete and hence is always present (and the only type of spectrum) 
of self-adjoint operators in finite dimensional Hilbert spaces. We refer the reader to the monographs \cite{AG93,BS87,D95,K66,RS72,RS78,Sch12} for more details on the spectrum of
self-adjoint operators.

The main objective of this expository paper is to illustrate and (partly) review how one can explicitely construct rough domains, 
singular potentials, or nonstandard boundary conditions
such that the essential spectrum of
a Schr\"{o}dinger operator coincides with a predefined subset of the real line. 
The closely connected problem to construct Schr\"{o}dinger operators
with predefined discrete spectrum is also briefly discussed. Very roughly speaking, 
the results in Section~\ref{sec1} are contained in the well-known papers \cite{A78,CdV87,HSS91,HKP97}, whereas the main results 
Theorem~\ref{th-BK+} and Theorem~\ref{essit} 
in the later sections seem to be new. 

More precisely, in Section~\ref{sec1} we treat Laplace operators subject to Neumann boundary conditions (Neumann Laplacians) on bounded domains.
It is often believed that self-adjoint Laplace-type operators on bounded domains always have purely discrete spectrum 
(or, equivalenty, a compact resolvent). This is indeed true 
for Laplace operators subject to Dirichlet boundary conditions (Dirichlet Laplacian), but, in general, not true for Neumann Laplacians. 
In fact, the discreteness of the spectrum of the Neumann Laplacian is equivalent to the compactness of the embedding $\H^1(\Omega)\hookrightarrow \L(\Omega)$,
and for this a necessary and sufficient criterion was obtained by C.J.~Amick \cite{A78}; cf. Theorem~\ref{th-A}. 
The standard example of a bounded domain for which essential spectrum for Neumann Laplacian appears is a so-called called 
``rooms-and-passages'' domain: a chain of bounded domains (``rooms'') connected through narrow rectangles  (``passages''), see Figure~\ref{fig1}. 
Rooms-and-passages domains are widely used in spectral theory and the theory of Sobolev spaces in order to demonstrate various 
peculiar effects (see, e.g., \cite{EH87,A78,Fr79}). {\color{black}Some spectral properties of such domains were investigated in \cite{BEW13}.   We also refer to the comprehensive monograph of V.G.~Mazya \cite{Ma11} (see also earlier contributions \cite{Ma79,Ma80,Ma85,MP97}), where rooms-and-passages together many other tricky domains were treated. }
In the celebrated paper \cite{HSS91} R.~Hempel, L.~Seco, and B.~Simon constructed a rooms-and-passages domain such that the spectrum of the 
Neumann Laplacian coincides with a prescribed closed set $\S\subset [0,\infty)$ with $0\in \S$. 
We review and prove their result in Theorem~\ref{th-HSS}; here also the continuous dependence of
the eigenvalues of Neumann Laplacians on varying domains discussed in Appendix~\ref{appa} plays an important role.
We also briefly recall another type of bounded domains -- so-called ``comb-like'' domains -- 
which allow to control the essential spectrum in the case $0\notin \S$.
Rooms-and-passages domains can also be used in a convenient way to control the discrete spectrum within compact intervals. 
We demonstrate this in Theorem~\ref{thCdV+}, where we establish a slightly weaker version of the following 
celebrated result by Y.~Colin de Verdi\`{e}re  \cite{CdV87}: for arbitrary numbers 
$0=\lambda_1<\lambda_2<\dots<\lambda_m$ there exists a bounded domain $\Omega\subset\R^n$ such that the spectrum of the
Neumann Laplacian on $\Omega$ is purely discrete and its first $m$ eigenvalues coincide with the above numbers. 
One of the main ingredients in our proof is a multidimensional version of the intermediate value theorem 
by R.~Hempel, T.~Kriecherbauer, and P.~Plankensteiner in \cite{HKP97}.
In fact, our Theorem~\ref{thCdV+} is also contained in a more general result established in \cite{HKP97}, where a domain was constructed in such a way 
that the essential spectrum and a part of the discrete spectrum of the Neumann Laplacian coincides with prescribed sets.

In Section~\ref{sec2} we show that similar tools and techniques can be used for a class of singular Schr\"odinger operators describing
the motion of quantum particles in potentials being supported at a discrete set. 
These operators are known as \textit{solvable models} of quantum mechanics \cite{AGHH05}.
Namely, we will treat differential operators defined by  the formal expression
\begin{gather*}
-{\d^2\over \d z^2}+\suml_{k\in\N}\be_k\langle\cdot\,,\,\delta_{z_k}'\rangle\delta_{z_k}', 
\end{gather*}
where $\delta_{z_k}'$ is the distributional derivative of the  delta-function supported at $z_k$,   $\langle\phi,\delta_{z_k}'\rangle$ denotes its action 
on the test function $\phi$ and $\be_k\in \R\cup\{\infty\}$. 
Such operators are called  Schr\"odinger operator with $\delta'$-interactions (or \textit{point dipole interactions}) and were studied (also in the multidimensional
setting) in numerous
papers; here we only refer the reader to \cite{GHKM80,GH87,BEL14,BGLL15,BLL13,ER16,JL16,AN06,AEL94,BK15,BN13a,BN13b,BSW95,CK19,EKMT14,Ex95a,Ex95b,
Ex96,EJ13,EJ14,EK15,EK18,EL18,KM10a,KM10b,KM14,LR15,
MPS16,M96,ZSY17} and the references therein.
We will show in Theorem~\ref{th-BK} (see also Theorem~\ref{th-BK+}) that the points $z_k$ and coefficients $\be_k$ can be chosen in such a way that the essential spectrum of 
the above operator coincides with a predefined closed set. In our proof we make use of well-known convergence results for quadratic forms, which we briefly recall in 
Appendix~\ref{appb}. Some of our arguments are also based and related to results in the recent paper 
\cite{KM14} by A.~Kostenko and M.M.~Malamud.

Finally, in Section~\ref{sec3} we consider a slightly more abstract problem which can also be viewed as a generalization of some of the above
problems: for a given densely defined symmetric operator $S$ with infinite defect numbers, that is, $S$ admits a self-adjoint extension 
$A$ and $\dom(A)/\dom(S)$ is infinite dimensional, 
and under the assumption that there exists a self-adjoint extension with discrete spectrum (or, equivalently, compact resolvent),
we construct a self-adjoint extensions of $S$ with prescribed essential spectrum (possibly unbounded from below and above). Here the prescribed essential spectrum is generated via a perturbation argument
and a self-adjoint operator $\Xi$ that acts in an infinite dimensional {\it boundary space} and plays the role of a parameter in a boundary condition.
Our result is also related to the series of papers \cite{ABMN05,ABN98,B04,BMN06,BN95,BN96,BNW93} by S.~Albeverio, J.~Brasche, M.M.~Malamud, H.~Neidhardt, and J.~Weidmann
in which the existence of self-adjoint extensions 
with prescribed point spectrum, absolutely continuous spectrum, and singular continuous spectrum in spectral gaps of a fixed underlying symmetric operator 
was discussed.

\subsubsection*{Acknowledgements}
A.K. is supported by the Austrian Science Fund (FWF) under Project No. M~2310-N32.

\section{Essential and discrete spectra of Neumann Laplacians\label{sec1}}

The main objective of this section is to highlight some spectral properties of the Neumann Laplacian on bounded domains. 
In the following let $\Omega$ be a bounded domain in $\mathbb{R}^n$ and assume that $n\ge 2$. As usual the Hilbert space of (equivalence classes of) 
square integrable complex functions on $\Omega$ is denoted by  $\L(\Omega)$, and $\H^1(\Omega)$ denotes the first order Sobolev space consisting 
of functions in $\L(\Omega)$ that admit weak derivatives (of first order) in $\L(\Omega)$. An efficient method to introduce the Neumann Laplacian
in a mathematically rigorous way is to consider the sesquilinear form $\a_\Omega$ defined by 
\begin{equation}\label{formn}
\a_\Omega[u,v]=\int_\Omega\nabla u\cdot\overline{\nabla v}\,\d x,\qquad \dom(\a_\Omega)=\H^1(\Omega).
\end{equation}
It is clear that this form is densely defined in $\L(\Omega)$, nonnegative, and one can show that the form is closed, i.e.
the form domain $\H^1(\Omega)$ equipped with the scalar product $\a_\Omega[\cdot,\cdot]+(\cdot,\cdot)_{\L(\Omega)}$ is complete.
The well-known first representation
theorem (see, e.g. \cite[Chapter 6, Theorem 2.1]{K66}) associates a unique nonnegative 
self-adjoint operator $A_\Omega$ in $\L(\Omega)$ to the form $\a_\Omega$ such that the domain inclusion 
$\dom(A_\Omega)\subset\dom(\a_\Omega)$ and the equality
\begin{equation}\label{opn}
(A_\Omega u, v)_{\L(\Omega)}=\a_\Omega[u,v],\quad u\in \dom(A_\Omega),\,\,v\in\dom(\a_\Omega),
\end{equation}
hold. The operator $A_\Omega$ is called \textit{the Neumann Laplacian} on $\Omega$. One can show  that
\begin{itemize}
\item $A_\Omega u = -\Delta u$,  where $-\Delta u$ is understood as a distribution.  

\item $\dom(A_\Omega)\subset \H^2_{\rm loc}(\Omega)\cap \H^1(\Omega)$.

\item If $\partial\Omega$ is  $C^2$-smooth then 
$$\dom(A_\Omega)=\left\{u\in  \H^2 (\Omega):\ \ds{\partial_n u}\!\restriction_{\partial\Omega}=0\right\},$$ 
where $\partial_n$ denotes the normal derivative on  $\partial\Omega$.
\end{itemize} 

Typically the boundary condition $\ds{\partial_n u}\!\restriction_{\partial\Omega}=0$ is referred to as {\it Neumann boundary condition},
which also justifies the terminology Neumann Laplacian. However, note that some regularity for the boundary of the domain has to
be required in order to be able to deal with a normal derivative. For completeness, we note that 
the assumption of a $C^2$-boundary above is not optimal (but almost) for $\H^2$-regularity
of the domain of the Neumann Laplacian.

The rest of this section deals with some spectral properties of Neumann Laplacians. First of all we discuss in a preliminary 
situation that
the Neumann Laplacian may have essential spectrum; since the domain $\Omega$ is bounded this may be a bit surprising at first sight.
In this context we then recall a well-known result due to R.~Hempel, L.~Seco, and B.~Simon from \cite{HSS91} how to explicitely 
construct a bounded rooms-and-passages-type domain (with nonsmooth boundary) such that the essential spectrum of the Neumann Laplacian coincides with a prescribed
closed set. Another related topic is to construct Neumann Laplacians on appropriate domains such that finitely many 
discrete eigenvalues coincide with a given set of points. Here we recall a famous result due to Y. Colin de Verdi\`{e}re from \cite{CdV87},
and supplement this theorem with a similar result which is  proved with a simple rooms-and-passages-type strategy.
Actually, Theorem~\ref{thCdV+} is also a special variant of a more general result 
by R.~Hempel, T.~Kriecherbauer, and P.~Plankensteiner in \cite{HKP97}.  

\subsection{Neumann Laplacians may have nonempty essential spectrum\label{subsec11}} 
 
Let $A_\Omega$ be the self-adjoint Neumann Laplacian in $\L(\Omega)$ and denote by $\sigma_{\ess}(A_\Omega)$ the essential spectrum
of $A_\Omega$.
It is well-known that $\sigma_{\ess}(A_\Omega)=\varnothing$ 
(which is equivalent to the compactness of the resolvent of $A_\Omega$ in $\L(\Omega)$)
if and only if 
\begin{gather}\label{compact}
\text{the embedding }i_\Omega:\H^1(\Omega)\hookrightarrow \L(\Omega)\text{ is compact};
\end{gather}   
cf. \cite[Satz 21.3]{T72}.
If the boundary $\partial\Omega$ is sufficiently
regular (for example Lipschitz) then \eqref{compact} holds; this result is known as Rellich's embedding theorem.
However, in general the embedding $i_\Omega$ need not be compact. In fact,
C.J.~Amick established in \cite[Theorem~3]{A78} a necessary and sufficient criterion for the compactness of the embedding operator $i_\Omega$, which we recall
in the next theorem. 

\begin{theorem}[Amick, 1978]\label{th-A}
The embedding $i_\Omega$ in \eqref{compact} is compact if and only if   
\begin{equation}\label{gamma}
\Gamma_\Omega:=\liml_{\eps\to 0}\sup_{u\in \H^1(\Omega)}{\|u\|^2_{\L(\Omega\e)}\over \|u\|^2_{\H^1(\Omega)}}=0,
\end{equation}
where $\Omega\e=\{x\in\Omega:\ \dist(x,\partial\Omega)<\eps\}$.
\end{theorem}

\begin{remark}\color{black}
	Necessary and sufficient conditions for $\sigma_{\ess}(A_\Omega)=\varnothing$ have been also obtained in \cite{Ma68}. These conditions are formulated in terms of capacities.
\end{remark}	

In \cite{A78} an example of a bounded domain $\Omega\subset\mathbb{R}^2$ consisting of countably 
many rooms $R_k$ and passages $P_k$ with $\Gamma_\Omega>0$ was constructed (see   Figure~\ref{fig1}). 

  \begin{figure}[h]
  \begin{picture}(350,78)
  \includegraphics[width=120mm]{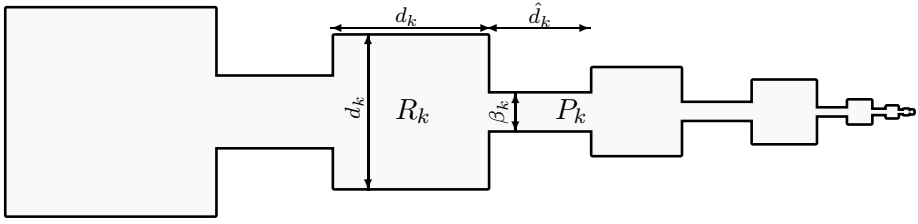}

  \put(-145,72){\vector(1,0){23}}
  \put(-145,72){\vector(-1,0){16}}
  \put(-145,78){$_{\hat d_k}$}

  \put(-190,72){\vector(-1,0){29}}
  \put(-190,72){\vector(1,0){31}}
  \put(-195,76){$_{d_k}$}

  \put(-135,38){${P_k}$}
  \put(-150,40){\vector(0,1){8}}
  \put(-150,40){\vector(0,-1){7}}
  \put(-156,36){\begin{rotate}{90}$_{\beta_k}$\end{rotate}}

  \put(-195,38){$R_k$}
  \put(-205,40){\vector(0,1){30}}
  \put(-205,40){\vector(0,-1){29}}
  \put(-210,38){\begin{rotate}{90}$_{d_k}$\end{rotate}}

  \end{picture}
  \caption{Rooms-and-passages domain $\Omega$ \label{fig1}}
  \end{figure}

For the convenience of the reader we wish to recall this construction in the following.
Note that we impose slightly different assumptions on the rooms and passages compared to \cite{A78}.
Consider some sequences
$(d_k)_{k\in\N}$ and 
$(\hat d_k)_{k\in\N}$  of positive numbers such that
\begin{gather}
\label{sum}
\suml_{k\in\N}d_k<\infty
\end{gather}
and assume that there is a constant $C_1>0$ with the property 
\begin{gather}
\label{ddd}
 \hat d_{k}\leq C_1\min\left\{d_k;\,d_{k+1}\right\}.
\end{gather}
Note that \eqref{sum}-\eqref{ddd} imply 
\begin{equation}\label{sum+}
\suml_{k\in\N}\hat d_k<\infty
\end{equation}
and
\begin{equation}\label{ddTo0}
\lim_{k\to\infty}d_k=0,\quad \lim_{k\to\infty}\hat d_k=0.
\end{equation}
One can choose, for example,
$d_k=(2k-1)^{-2},\ \hat d_k=(2k)^{-2},\  k\in\N$;
then conditions \eqref{sum} and \eqref{ddd} hold with $C_1\in \big[{9\over 4},\infty\big)$. 
Finally, let  
$(\beta_k)_{k\in\N}$ be a sequence of positive numbers
such that for all $k\in\N$
\begin{gather}\label{beta}
\beta_k \leq C_2 (\hat d_k)^\al
\end{gather}
with some  $\al\ge 3$ and $C_2>0$ such that
\begin{gather}
\label{C2}
C_2\leq \frac{1}{C_1}\cdot\left(\max_{k\in\N}\hat d_k\right)^{1-\alpha}. 
\end{gather}

In the next step define the sequence $(x_k)_{k\in\N}$ by 
\begin{equation}\label{xk}
x_k:=\suml_{j=1}^k (d_j + \hat d_j) - \hat d_k,
\end{equation}
and define the rooms $R_k$ and passages $P_k$  by
\begin{equation}\label{RP1}
 R_k: = (x_k-d_k,x_k)\times \left(-{d_k\over 2},{d_k\over 2}\right)
\end{equation}
and
\begin{equation}\label{RP2}
P_k: = \bigl[x_k,x_{k}+\hat d_k\bigr]\times \left(-{\beta_k\over 2},{\beta_k\over 2}\right),
\end{equation}
respectively. Finally, the union of $R_k$ and $P_k$ leads to the desired rooms-and-passages domain 
\begin{equation}\label{RP3}
 \Omega:=\cupl_{k\in\N}\left(R_k\cup P_k\right).
\end{equation}
From \eqref{sum}, \eqref{sum+}, and \eqref{beta} it is clear that $\Omega$ is bounded.
Using \eqref{ddd}, \eqref{beta}, \eqref{C2} and taking into account that $\alpha>3$ we obtain the estimate 
\begin{gather}\label{betaestk}
\beta_k \leq C_2 (\hat d_k)^\al  \leq C_2(\hat d_k)^{\al-1}C_1 \min\left\{d_k,d_{k+1}\right\}   \leq \min\left\{d_k,d_{k+1}\right\}.
\end{gather}
Hence the thickness of the passage $P_k$ is not larger than the sides 
of the adjacent rooms  $R_{k}$ and $R_{k+1}$, which also shows that $\Omega$ is indeed an open set.

It will now be illustrated that for this particular domain $\Omega$ the quantity $\Gamma_\Omega$ in \eqref{gamma} is positive, so that
the embedding in \eqref{compact} is not compact. In particular, the essential spectrum of the Neumann Laplacian in $\L(\Omega)$ is not empty.
For this purpose consider the piecewise linear functions $u_k$, $k=2,3,\dots$, defined by 
\begin{gather*}
u_k(\textbf{x})=
\begin{dcases}
\ds{1\over d_k},&\textbf{x}=(x,y)\in R_k,\\ 
\ds{x_k+\hat d_k-x\over d_k\hat d_k},&\textbf{x}=(x,y)\in P_k,\\ 
\ds{x_{k-1}-x\over d_k(x_{k-1}-x_k+d_k)},&\textbf{x}=(x,y)\in P_{k-1},\\ 
0,&\text{otherwise}.
\end{dcases}
\end{gather*}
Note that $x_{k-1}-x_k+d_k=-\hat d_{k-1}$ by \eqref{xk}. It is easy to see that
the function $u_k$ belongs to $\H^1(\Omega)$. Next we evaluate its $\L$-norm. One computes
\begin{multline}\label{uk0}
\|u_k\|^2_{\L(\Omega)}=
\|u_k\|^2_{\L(R_k)}+\|u_k\|^2_{\L(P_{k-1})}+\|u_k\|^2_{\L(P_{k})}
\\=
1+ {1\over 3(d_k)^2}\left(\beta_{k-1}\hat d_{k-1}+\beta_{k}\hat d_k\right)
\leq
1+{C_2 \over 3(d_k)^2}\left( (\hat d_{k-1})^{\alpha+1}+(\hat d_{k})^{\alpha+1} \right).
\end{multline}
We also have (cf.~\eqref{ddd})
\begin{equation}
\label{uk00}
\hat d_{k-1} \leq C_1 d_k\quad\text{and}\quad
\hat d_{k  } \leq C_1 d_k.
\end{equation}
Using \eqref{uk00} and taking into account that
$\lim_{k\to\infty} d_k = 0$ and $\alpha\geq 3$, we obtain from \eqref{uk0}:
\begin{equation}\label{uk}
\|u_k\|^2_{\L(\Omega)}=1+o(1)\text{ as }k\to\infty.
\end{equation}
Now we estimate the $\L$-norm of $\nabla u_k$. Using \eqref{beta} and \eqref{uk00} we get
\begin{equation}\label{estixx}
\begin{split}
\|\nabla u_k\|^2_{\L(\Omega)}&=
\|\nabla u_k\|^2_{\L(P_{k-1})}+\|\nabla u_k\|^2_{\L(P_{k})}
=
{1\over (d_k)^2}\left({\beta_{k-1}\over \hat d_{k-1}}+{\beta_k\over \hat d_k}\right)
\\
&\leq
 {C_2 \over (d_k)^2}\left( (\hat d_{k-1})^{\alpha-1}+(\hat d_{k})^{\alpha-1} \right)\leq
2C_1 C_2(d_{k})^{\alpha-3}.
\end{split}
\end{equation} 
Moreover, it is clear that for any $\eps>0$ there exists $k(\eps)\in\N$ such that 
$\mathrm{supp}(u_k)\subset\Omega\e$ for all $k\ge k(\eps)$. This, \eqref{uk}, \eqref{estixx}, and 
\eqref{ddTo0}
yield $\Gamma_\Omega>0$ (recall that $\alpha\geq 3$). 
As an immediate consequence we conclude the following corollary.

\begin{corollary}\label{cor1}
Let $\Omega$ be the bounded rooms-and-passages domain in \eqref{RP3} and let $A_\Omega$ be the self-adjoint Neumann Laplacian in $\L(\Omega)$.
Then
$$\sigma_{\ess}(A_\Omega)\not=\varnothing $$
\end{corollary}

The natural question that arises in the context of Corollary~\ref{cor1} is what form the essential spectrum of the Neumann Laplacian $A_\Omega$ may have. 
This topic is discussed in the next subsection.

\subsection{Neumann Laplacian with prescribed essential spectrum\label{subsec12}}

In the celebrated  paper \cite{HSS91} R.~Hempel, L.~Seco, and B.~Simon 
have shown, using rooms-and-passages-type domains of similar form as above, that the 
essential spectrum of the Neumann Laplacian can be rather arbitrary. 
Below we briefly describe their construction.

We fix sequences  $(d_k)_{k\in\N}$ and $(\hat d_k)_{k\in\N}$ of positive numbers
satisfying 
\begin{gather}
\label{dd}
\suml_{k\in\N}(d_k+\hat d_k)<\infty,
\end{gather}
and of course one then has \eqref{ddTo0}.
Let the domain $\Omega\subset\mathbb{R}^2$ consist  of countably many rooms $R_k$ and passages $P_k$. Furthermore, 
in each room we insert an additional ``wall'' $W_k^{\al_k}$ (see   Figure~\ref{fig2}), and the resulting modified room is then denoted by
$R_k^{\alpha_k}=R_k\setminus W_k^{\al_k}$, so that 
\begin{equation}\label{omega23}
\Omega=\cupl_{k\in\mathbb{N}}\bigl( R_k^{\al_k}\cup P_k\bigr)=\cupl_{k\in\mathbb{N}}\bigl((R_k\setminus W_k^{\al_k})\cup P_k\bigr).
\end{equation}
Here $R_k$ and $P_k$ are defined by \eqref{RP1} and \eqref{RP2}
with $\beta_k$ satisfying
\begin{gather}\label{b} 
0<\beta_k\leq \min\left\{d_k;\,d_{k+1}\right\};
\end{gather}
cf. \eqref{betaestk}.
The walls
$W_k^{\al_k}$ are given by 
$$W_k^{\al_k}:=\left\{\mathbf{x}=(x,y)\in\mathbb{R}^2:\ x=x_k-{d_k\over 2},\ |y|\in \bigg[{\al_k\over 2},{d_k\over 2}\bigg]\right\},$$
where it is assumed that the sequence $(\al_k)_{k\in\N}$ satisfies
\begin{gather}\label{a} 
0<\al_k\leq d_k.
\end{gather}
 
 \begin{figure}[h]
 \begin{picture}(290,70)
 \includegraphics[width=100mm]{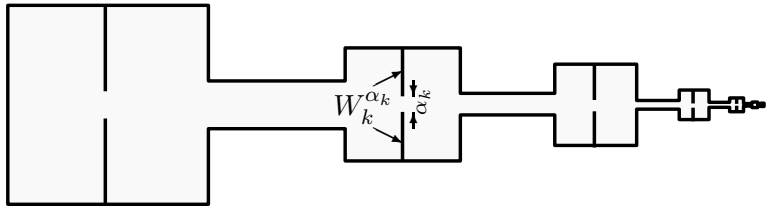}
  
 \put(-162,35){$W_k^{\al_k}$}
 \put(-147,30){\vector(2,-1){11}}
 \put(-147,46){\vector(2, 1){11}}
 
  \put(-132,30){\vector(0,1){6}}
  \put(-132,48){\vector(0,-1){6}}
  \put(-129,35){\begin{rotate}{90}$_{\al_k}$\end{rotate}}
  
  \end{picture}
  \caption{Rooms-and-passages domain with additional walls\label{fig2}}
  \end{figure}
  
\begin{theorem}[Hempel-Seco-Simon, 1991]\label{th-HSS}
Let $\S\subset [0,\infty)$ be an arbitrary closed set such that  $0\in \S$. Then there exist  sequences $(d_k)_{k\in\N}$, $(\hat d_k)_{k\in\N}$, $(\al_k)_{k\in\N}$,
$(\be_k)_{k\in\N}$ satisfying \eqref{dd}, \eqref{b}, and \eqref{a} such
that 
\begin{gather}
\label{essS}
\sigma_\ess(A_\Omega)=\S.
\end{gather}
\end{theorem}  

\begin{proof}
This sketch of the proof from \cite{HSS91} consists of three steps. In the first step, where we skip a perturbation argument, it is shown that
the original problem to ensure \eqref{essS} for the Neumann Laplacian $A_\Omega$ on the domain $\Omega$ 
can be reduced to show the same property for a ``decoupled'' Neumann Laplacian $A_{\rm dec}$. The spectrum of this operator can be described explicitly, 
which is done in the second step. Finally, in a third step the parameters are adjusted in such a way that \eqref{essS} holds.\\
 
\noindent
{\it Step 1.} Let $(d_k)_{k\in\N}$, $(\hat d_k)_{k\in\N}$, $(\al_k)_{k\in\N}$, and
$(\be_k)_{k\in\N}$ be some sequences that satisfy \eqref{dd}, \eqref{b}, and \eqref{a}. Let $\Omega$ be the corresponding 
domain in \eqref{omega23} and let $A_\Omega$ be the Neumann Laplacian on $\Omega$ defined via the quadratic form as in \eqref{formn}--\eqref{opn}.
In the following we denote by $A_{R_k^{\al_k}}$ 
the Neumann Laplacian
on the domain $R_k^{\alpha_k}=R_k\setminus W_k^{\al_k}$, also defined via the quadratic form
\begin{equation}\label{form5}
\a_{R_k^{\alpha_k}}[u,v]=\int_{R_k^{\al_k}}\nabla u\cdot\overline{\nabla v}\,\d x,\quad
\dom\bigl(\a_{R_k^{\al_k}}\bigr)=\H^1(R_k^{\al_k}),
\end{equation}
in the same way as in \eqref{formn}--\eqref{opn}. Informally speaking, the functions in the domain of this operator satisfy Neumann 
boundary conditions on the boundary of the room $R_k$ and, in addition, 
Neumann boundary conditions on both sides of the additional wall $W_k^{\al_k}$. Furthermore, we will
make use self-adjoint Laplacians  
on the interiors $\mathring P_k$ of the passages $P_k$ with mixed Dirichlet and Neumann boundary conditions. 
More precisely, $A_{\mathring P_k}^{\text{\tiny{DN}}}$
denotes the self-adjoint Laplacian defined on a subspace of $\H^1(\mathring P_k)$, where it is assumed that the functions in the domain satisfy 
Neumann boundary conditions on $\{\textbf{x}=(x,y)\in\partial \mathring P_k:\ y=\pm \be_k/2\}$
and Dirichlet boundary conditions on the remaining part $\{\textbf{x}=(x,y)\in\partial \mathring P_k:\ x=x_k\vee x=x_{k}+\hat d_k\}$ of the boundary.
Now  consider the ``decoupled'' operator 
\begin{equation}\label{orto}
A_{\rm dec}=
\bigoplus_{k\in\N} 
\bigl(
A_{R_k^{\al_k}} 
\oplus
A_{\mathring P_k}^{\text{\tiny{DN}}}\bigr)
\end{equation}
as an orthogonal sum of the self-adjoint operators $A_{R_k^{\al_k}}$ and $A_{\mathring P_k}^{\text{\tiny{DN}}}$ in the space 
$$\L(\Omega)=\bigoplus_{k\in\N} \bigl(\L(R_k^{\al_k})\oplus\L(\mathring P_k)\bigr).$$
Then one can show that the resolvent difference 
$(A_{\rm dec}+\Id)^{-1} - (A_\Omega+\Id)^{-1}$
of the Neumann Laplacian $A_\Omega$ and the decoupled operator $A_{\rm dec}$
is a compact  operator
provided $\beta_k\to 0$ sufficiently fast as $k\to \infty$.  
We fix such a sequence $(\be_k)_{k\in\N}$; then from
Weyl's theorem  (see, e.g., \cite[Theorem XIII.14]{RS78}) one concludes
$$\sigma_\ess(A_\Omega)=\sigma_\ess(A_{\rm dec})$$ and hence it remains to show that 
$$\sigma_\ess(A_{\rm dec})=\S.$$

\noindent
{\it Step 2.}
First we shall explain how the eigenvalues of the Neumann Laplacian on $R_k^{\al_k}$ depend on the size of the wall $W_k^{\al_k}$
inside $R_k$. In this step of the proof the value $\al_k=0$ is also allowed (in this case the room $R_k^{\al_k}$ decouples: it becomes 
a union of two disjoint rectangles). 
We denote the eigenvalues of $A_{R_k^{\alpha_k}}$ (counted with multiplicities) and ordered
as a nondecreasing sequence by $(\lambda_{j}(R_k^{\alpha_k}))_{j\in\N}$.
It is not difficult to check that the corresponding forms $\a_{R_k^{\al_k}}$ in \eqref{form5} are monotone in the parameter $\alpha_k$, that is, 
for $0\leq\alpha_k\leq\widetilde \alpha_k\leq d_k$ one has 
\begin{gather*}
\dom(\a_{R_k^{\al_k}})\supset\dom(\a_{R_k^{\widetilde \al_k}}),\\
\a_{R_k^{\al_k}}[u,u]= \a_{R_k^{\widetilde \al_k}}[u,u]\text{ for all }u\in\dom(\a_{R_k^{\widetilde \al_k}}),
\end{gather*}
which means $\a_{R_k^{\al_k}}\leq \a_{R_k^{\widetilde \al_k}}$ in the sense of ordering of forms.
Then it follows from the min-max principle (see, e.g., \cite[Section~4.5]{D95}) that for each $j\in\dN$ the function
\begin{gather}\label{alpfa-f}
[0,d_k]\supset\alpha_k\mapsto\lambda_j(R_k^{\al_k}) 
\end{gather}
is  nondecreasing and by Theorem~\ref{th-contin}
this function is also continuous. In the present situation it is clear that
\begin{equation}\label{lambda1R}
\lambda_1(R_k^{\al_k})=0,
\end{equation}
and
\begin{equation}
\label{lambda2R}
\lambda_2(R_k^{\al_k})=
\begin{cases}
0,&\al_k=0,\\
\left({\pi /   d_k}\right)^2,&\al_k=d_k,
\end{cases}
\end{equation} 
and due to the monotonicity of the function \eqref{alpfa-f} one also has 
\begin{gather}\label{lambda3R}
\lambda_3(R_k^{\al_k})\geq \lambda_3(R_k^0)=\left({\pi /   d_k}\right)^2.
\end{gather}
Furthermore, if $(\mu_j(\mathring P_k))_{j\in\N}$ denote the eigenvalues (counted with multiplicities) of $A_{\mathring P_k}^{\text{\tiny{DN}}}$ ordered
as a nondecreasing sequence then one verifies that the first eigenvalue $\mu_1(\mathring P_k)$ is given by 
\begin{gather}\label{lambda1mu1b}
\mu_1(\mathring P_k)= ({\pi / \hat d_k} )^2
\end{gather} 
for all $k\in\dN$.
From the orthogonal sum structure in \eqref{orto} it is clear that
$$
\sigma_{\ess}(A_{\rm dec})=
\acc\big((\lambda_j(R_k^{\al_k}))_{j,k\in\N}\big)\cup\acc\big((\mu_j(\mathring P_k))_{j,k\in\N}\big),
$$
where the symbol $\acc$ denotes the set of accumulation points of a sequence.
Observe that the eigenvalues $(\mu_j(\mathring P_k))_{j\in\N}$ do not have any finite accumulation point (the smallest eigenvalue satisfies 
\eqref{lambda1mu1b} and $\lim_{k\to\infty }\hat d_k= 0$ by assumption) and hence we obtain
$$
\sigma_{\ess}(A_{\rm dec})=
\acc\big((\lambda_j(R_k^{\al_k}))_{j,k\in\N}\big).
$$

\noindent
{\it Step 3.} Now we complete the proof by adjusting the parameters in the above construction. 
Since by assumption $\S$ is a closed subset of $[0,\infty)$ one can always find a sequence $(s_k)_{k\in \N}$
such that 
\begin{gather}
\label{s-assumpt}
s_k> 0\text{ and }\acc((s_k)_{k\in \N})=
\begin{cases}
\S\setminus\{0\},&0\text{ is an isolated point of }\S,\\
\S,&\text{otherwise}.
\end{cases}
\end{gather}
Next, for each $k\in\N$ we fix a number $d_k>0$ such that
\begin{gather}\label{sd}
s_k < (\pi/ d_k)^2.
\end{gather} 
In addition, we assume that  the numbers $d_k$ are chosen small enough so that
$$\suml_{k\in\N}d_k<\infty.$$
We also fix a sequence of positive numbers $(\hat d_k)_{k\in\mathbb{N} }$ such that 
$$\suml_{k\in\N}\hat d_k<\infty.$$
Using
the continuity of the function \eqref{alpfa-f} 
and taking into account \eqref{lambda2R} and \eqref{sd} it is clear that 
there exists $\al_k\in (0,d_k)$ such that the second eigenvalue $\lambda_2(R_k^{\al_k})$ of the Neumann Laplacian $A_{R_k^{\al_k}}$ 
satisfies
\begin{gather}\label{lambda2}
\lambda_2(R_k^{\al_k})=s_k
\end{gather}
for all $k\in\dN$.
Furthermore, by construction we also have $\lambda_1(R_k^{\al_k})=0$ and $\lambda_j(R_k^{\al_k})\geq (\pi/d_k)^2$ for $j\geq 3$ 
(cf. \eqref{lambda3R}), and hence we conclude together with
$\lim_{k\to\infty }d_k= 0$, \eqref{lambda1R}, \eqref{s-assumpt},  \eqref{lambda2},
and $0\in \S$
that 
$$\acc\big((\lambda_j(R_k^{\al_k}))_{j,k\in\N}\big)=
\acc\big((\lambda_j(R_k^{\al_k}))_{j\leq 2,\,k\in\N}\big)=
\{0\}\cup
\acc((s_k)_{k\in\N}) = \S.$$ 
This implies $\sigma_\ess (A_{\rm dec})=\S$ and completes the proof.  
\end{proof}

Besides \eqref{essS} it is also shown in \cite{HSS91}  that 
the absolutely continuous spectrum $\sigma_{\rm ac}(A_\Omega)$ of $A_\Omega$ is empty.
The argument is as follows: It is verified that the difference
$(A_\Omega+\Id)^{-2}-(A_{\rm dec}+\Id)^{-2}$
is a trace class operator, and consequently the absolutely continuous spectra of 
$A_\Omega$ and $A_{\rm dec}$ coincide; cf. \cite[page~30,\,Corollary~3]{RS79}.
Since $\sigma(A_{\rm dec})$ is  pure point one concludes $\sigma_{\rm ac}(A_\Omega)=\sigma_{\rm ac}(A_{\rm dec})=\emptyset$.
Note, that the absolutely continuous spectrum of the Neumann Laplacian on a bounded domain is not always empty. 
For example, in \cite{Si92} B.~Simon constructed a bounded set $\Omega$ having the form of a ``jelly roll'' 
such that $\sigma_{\rm ac}(A_\Omega)=[0,\infty)$.

In \cite{HSS91} R.~Hempel, L.~Seco, and B.~Simon 
also constructed a domain for which \eqref{essS} holds
without the restriction $0\in \S$.
For this purpose so-called comb-like domains are used, see Figure~\ref{fig3}.
To  construct the comb one attaches  a  sequence of ``teeth'' $(T_k^{\al_k})_{k\in\N}$ to a fixed rectangle $Q$;
the tooth $T_k^{\al_k}$ is obtained from a rectangle $T_k$ by removing an
internal wall  $W_k^{\al_k}$. The teeth have bounded lengths,  shrinking widths, and are stacked together without gaps.

The analysis is similar to the rooms-and-passages case. One can prove that the Neumann Laplacian $A_\Omega$ on such a comb-like domain $\Omega$ 
is a compact perturbation of the decoupled operator
$$
A_{\rm dec}=
\left(\bigoplus_{k\in\N} 
A_{T_k^{\al_k}}^{\text{\tiny{DN}}}\right) \oplus A_Q,
$$
where 
$A_{Q}$ is the Neumann Laplacian on $Q$ and $A_{T_k^{\al_k}}^{\text{\tiny{DN}}}$ is the Laplace operator on the tooth $T_k^{\al_k} $ subject to 
Neumann boundary conditions on $\partial T_k^{\al_k}\setminus\partial Q$ 
and Dirichlet boundary conditions on $\partial T_k^{\al_k} \cap\partial Q$.
The walls $W_k^{\al_k}$   are adjusted in such a way that 
the lowest eigenvalue 
  of $A_{T_k^{\al_k}}^{\text{\tiny{DN}}}$ coincides with a predefined number $s_k$, while the next eigenvalues tend to $\infty$ as $k\to\infty$ and do not contribute to essential spectrum.
 \begin{figure}[h]
 \begin{picture}(130,165)
 \includegraphics[height=60mm]{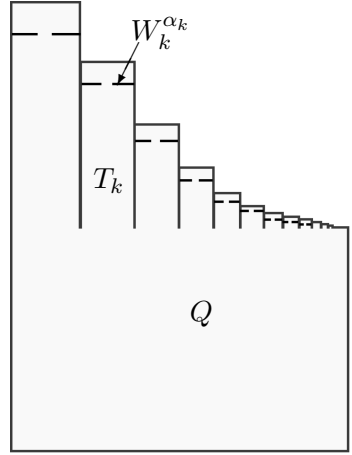}
 \put(-60,50){$Q$}
 \put(-96,100){$T_k$}
 \put(-82,156){$W_k^{\al_k}$}
 \put(-79,155){\vector(-1,-2){8}}
 \end{picture}
 \caption{Comb-like domain \label{fig3}}
 \end{figure}

The important and somewhat surprising element of the rooms-and-pass\-age and comb-like domain constructions
is the form of the decoupled operators with mixed Dirichlet and Neumann boundary conditions. 
The fact that one chooses Dirichlet conditions on the common part of the
boundaries of the passages $P_k$ and modified rooms $R_k^{\al_k}$, and similarly on the common part of the
boundaries of the modified teeth $T_k^{\al_k}$ and the rectangle $Q$
is due to the following well-known effect (see, e.g., \cite{An87,Arr95,J89}): 
the spectrum of the Neumann Laplacian on  
$Q\cup T\e$, where $Q$ is a fixed domain and $T\e$ is an  attached 
``handle'' of fixed length $L$ and width $\eps$, converges 
to the direct sum of the Neumann Laplacian
on $Q$ and the one-dimensional Dirichlet Laplacian on $(0,L)$ as $\eps\to 0$.
 
{\color{black}  
Finally, we note that A.A.~Kiselev and B.S.~Pavlov \cite{KP94} obtained Theorem~\ref{th-HSS} for
(a kind of) Neumann Laplacian on a bounded set consisting of an array of two-dimensional domains   connected 
by {intervals}.}

\subsection{Neumann Laplacian with prescribed discrete spectrum\label{subsec13}}

In this section we are interested in the discrete spectrum of Neumann Laplacians. First we recall
a result by Y.~Colin de Verdi\`{e}re from  \cite{CdV87}.

\begin{theorem}[Colin de Verdi\`{e}re, 1987]\label{thCdV}
Let $n\in\N\setminus\{1\}$ and assume that
\begin{gather*}
0=\lambda_1<\lambda_2<\dots<\lambda_m,\quad m\in\N,
\end{gather*}
are fixed numbers.  
Then there exists a bounded domain $\Omega\subset\R^n$ such that the spectrum of the
Neumann Laplacian 
$A_\Omega$ on $\Omega$ is purely discrete and 
\begin{gather*}
\lambda_k(\Omega)=\lambda_k,\quad k=1,\dots,m,
\end{gather*}
where $\left(\lambda_k(\Omega)\right)_{k\in\N}$ denotes the sequence of the eigenvalues of 
$A_\Omega$ numbered in increasing order with multiplicities taken into account.
\end{theorem}

In fact, the main result in \cite{CdV87} concerns Riemannian manifolds: for an arbitrary 
compact connected manifold one can construct a Riemannian metric on this manifold in such a way that
the first $m$ eigenvalues of the corresponding Laplace-Beltrami operator coincide with $m$ predefined numbers.
The idea of the proof of this theorem in \cite{CdV87} is to first 
construct a suitable differential operator $A_\Gamma$ on a metric graph $\Gamma$ such that the first $m$ eigenvalues of $A_\Gamma$ coincide 
with the predefined numbers $\lambda_k$, $k=1,\dots,m$. 
Then the graph $\Gamma$ is  ``blown''  up to a tubular thin domain $\Omega$ in such a way that the first $m$ eigenvalues of the Neumann Laplacian 
$A_\Omega$ on $\Omega$ are asymptotically close to the first $m$ eigenvalues of $A_\Gamma$ provided the cross-section of $\Omega$ tends to zero.
For dimensions $n\ge 3$  the above theorem is extended in \cite{CdV87} by allowing nonsimple eigenvalues. More precisely, 
one can construct a domain in $\R^n$, $n\ge 3$, such that the first $m$ eigenvalues of the Neumann Laplacian on this domain coincide with the predefined numbers 
$$0=\lambda_1<\lambda_2\le\lambda_3\le \dots\le \lambda_m.$$

A similar theorem for the Dirichlet Laplacian (at least for $n=2$) can not be expected. In fact, 
by a well-known result of L.E.~Payne, G.~P\'olya, and H.F.~Weinberger from \cite{PPW56} 
the ratio between the $k$-th and the $(k-1)$-th Dirichlet eigenvalue (for compact domains in $\mathbb{R}^2$) is bounded from 
above (by a domain independent constant) and hence the eigenvalues can not be placed arbitrarily on $(0,\infty)$.   
 
Y.~Colin de Verdi\`{e}re's result from above was later
improved by R.~Hempel, T.~Kriecherbauer, and P.~Plankensteiner in \cite{HKP97}, where  a  bounded domain $\Omega$ was constructed
such that the essential spectrum and a finite part of the discrete spectrum of
$A_\Omega$
coincide with predefined sets. In their construction comb-like domains were used; see Figure~\ref{fig3}.
\medskip 

The next theorem may be viewed as a variant of Theorem~\ref{thCdV}, although it is actually a slightly weaker version. In fact, we present this result here 
since it can be proved
with a similar rooms-and-passages domain strategy as Theorem~\ref{th-HSS}. An important ingredient is a multidimensional version
of the intermediate value theorem from \cite{HKP97}; cf. Lemma~\ref{lemma-hempel} below.

\begin{theorem} \label{thCdV+}
Let $n\in\N\setminus\{1\}$ and assume that
\begin{gather*}
0<\nu_1<\nu_2<\dots<\nu_m,\quad m\in\N,
\end{gather*}
are fixed numbers. 
Then there exists a bounded domain $\Omega\subset\R^n$ such that the spectrum of the
Neumann Laplacian 
$A_\Omega$ on $\Omega$ is purely discrete and
\begin{gather*}
\lambda_{m+k}(\Omega)=\nu_k,\quad k=1,\dots,m,
\end{gather*}
where $\left(\lambda_l(\Omega)\right)_{l\in\N}$ denotes the sequence of the eigenvalues of 
$A_\Omega$ numbered in increasing order with multiplicities taken into account.
\end{theorem}

\begin{proof}

We consider the case $n=2$; the construction and the arguments in dimensions $n\geq 3$ are similar and left to the reader. 
The proof consists of several steps and
in principle the strategy is similar to the one in the proof of Theorem~\ref{th-HSS}.
More precisely, we fix an arbitrary open and bounded interval $\I\subset (0,\infty)$  containing all the points $\nu_k$.
First we consider the decoupled domain consisting of $m$ pairwise disjoint rooms
$R_k^{\al_k}=R_k\setminus W_k^{\al_k}$ stacked in a row (as before $R_k$ is a square and $W_k^{\al_k}$ is an internal wall in it). The first eigenvalue of the 
Neumann Laplacian on $R_k^{\al_k}$ is zero. 
Moreover, under a suitable choice of $d_k$ and $\al_k$
the second eigenvalue coincides with $\nu_k$  and the third eigenvalue 
is
larger than $\sup\I$. Consequently the first $m$ eigenvalues of the Neumann Laplacian 
on $\cup_{k=1}^m R_k^{\al_k}$ are zero, the next $m$ eigenvalues are $\nu_1,\dots,\nu_m$, and all further eigenvalues are contained in $(\sup\I,\infty)$.
Afterwards we connect the rooms by small windows of length $\eps$ (see Figure~\ref{fig4}); the resulting domain is denoted by $\Omega$. If $\eps$ is small 
the spectrum changes slightly. Namely, 
\begin{itemize}
\item the eigenvalues $\lambda_1(\Omega),\dots,\lambda_m(\Omega)$ remain in $[0,\inf\I)$,
\item the eigenvalues $\lambda_{m+1}(\Omega),\dots,\lambda_{2m}(\Omega)$ remain in small neighborhoods
of $\nu_1,\dots,\nu_m$, respectively, moreover they still belong to $\I$,
\item the rest of the spectrum remains in $(\sup\I,\infty)$. 
\end{itemize} 
It remains to establish the coincidence of 
$\lambda_{k+m}(\Omega)$ and $\nu_k$ as $k=1,\dots,m$ (so far they are only close as $\eps>0$ is sufficiently small),   
which is done by using a multidimensional version of mean value theorem from \cite{HKP97}. Roughly speaking, the shift of the eigenvalues 
that appears after inserting the small windows can be compensated by varying
the constants $\al_1,\dots,\al_m$ appropriately. 
In the following we implement this strategy.
\medskip

\noindent
{\it Step 1.} Fix some positive numbers $d_1,\dots,d_m$ and let $\al_1,\dots,\al_m$, and $\eps$ be nonnegative numbers satisfying  
\begin{gather}\label{epsi}
\al_k\in [0,d_k]\quad\text{and}\quad
\eps\in [0,\min_{k=1,\dots,m} d_k].
\end{gather}
In the following we shall use the notation $\al:=\{\al_1,\dots,\al_m\}$.
We introduce the domain $\Omega^{\al,\eps}$ consisting of $m$ modified rooms $R_k^{\al_k}$ being stacked in a row and connected through a small window
$P_k^\eps$ of length $\eps$. Each room $R_k^{\al_k}$ is obtained by removing from a square with side length $d_k$   
two additional walls of length $(d_k-\al_k)/2$; see Figure~\ref{fig4}.
More precisely, let $x_k:=\sum_{j=1}^{k}d_j$, define the rooms by 
\begin{gather*}
R_k^{\al_k}=\left(( x_k- d_k,x_k)\times\left(-{d_k\over 2},{d_k\over 2}\right)\right) \setminus W_k^{\alpha_k},
\end{gather*}
where the walls are  
\begin{gather*}
W_k^{\alpha_k}=\left\{(x,y)\in\mathbb{R}^2:\ x=x_k-{d_k\over 2},\ |y|\in \bigg[{\al_k\over 2},{d_k\over 2}\bigg]\right\}
\end{gather*}
and the windows are $P_k^\eps=\left\{x_k\right\}\times \left(-{\eps\over 2},{\eps\over 2}\right)$. Then the domain $\Omega^{\al,\eps}$ that we will use is defined as
$$\Omega^{\al,\eps}=\left(\cupl_{k=1}^m R^{\al_k}_k \right)\cup\left(\cupl_{k=1}^{m-1} P^\eps_k\right).$$
 \begin{figure}[h]
 \begin{picture}(100,80)
 \includegraphics[width=50mm]{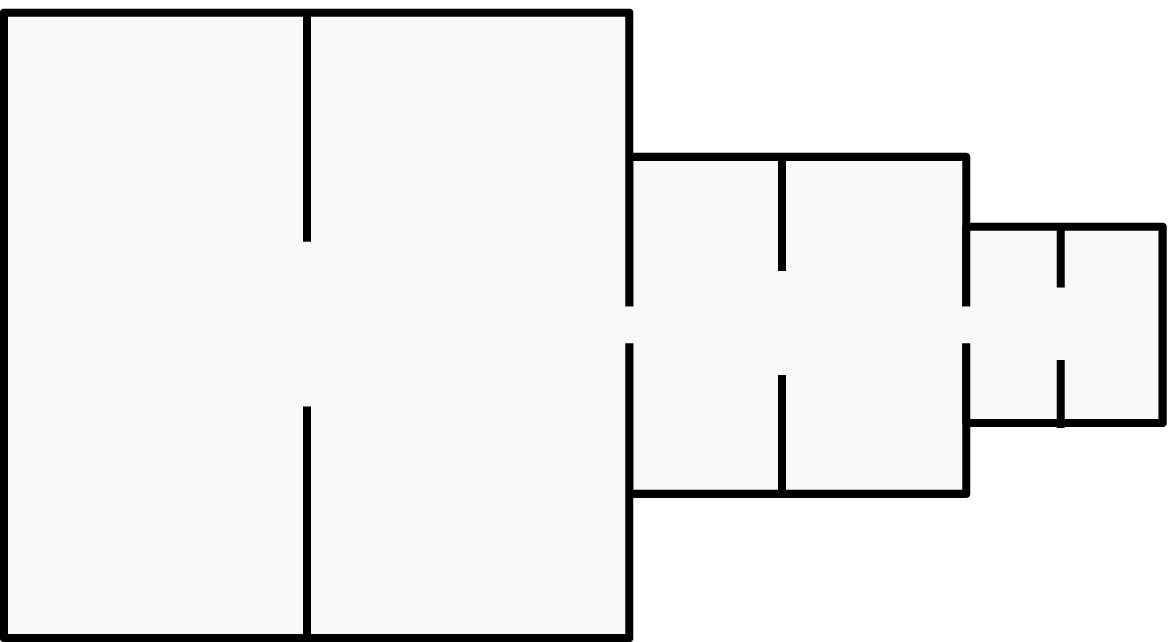}
 \put(-130,30){$R_1^{\al_1}$}
 \put(-100,36){\vector(0,-1){8}}
 \put(-100,35){\vector(0,1){14}}
 \put(-98,35){$\al_1$}
 
 \put(-63,46){\vector(0,-1){5}}
 \put(-63,32){\vector(0,1){5}}
 \put(-61,36){$\eps$}

 \end{picture}
 \caption{\label{fig4} Domain $\Omega^{\al,\eps}$ for $m=3$}
 \end{figure}

The Neumann Laplacian on $\Omega^{\al,\eps}$ will be denoted by $A_{\Omega^{\al,\eps}}$ and it is important to observe that for $\eps=0$
this self-adjoint operator in $\L(\Omega^{\al,\eps})$ decouples in a finite orthogonal sum of Neumann Laplacians $A_{R_k^{\al_k}}$ on the 
rooms $R_k^{\al_k}$, that is, 
one has 
\begin{equation}\label{orto+}
A_{\Omega^{\al,0}}=\bigoplus_{k=1}^m A_{R_k^{\al_k}}
\end{equation}
with respect to the corresponding space decomposition
$$
\L(\Omega^{\al,0})=\bigoplus_{k=1}^m \L(R_k^{\al_k}).
$$
Note that $\L(\Omega^{\al,0})=\L(\Omega^{\al,\eps})$ and hence $A_{\Omega^{\al,\eps}}$ and $A_{\Omega^{\al,0}}$ act in the same space for all $\eps$ in \eqref{epsi}. 

It is clear from \eqref{orto+} that the spectrum of the decoupled operator $A_{\Omega^{\al,0}}$ is the union of the
spectra of the Neumann Laplacians $A_{R_k^{\al_k}}$, $k=1,\dots,m$. 
We recall (see Step~2 of the proof of Theorem~\ref{th-HSS}) that the functions 
\begin{equation}\label{incre}
[0,d_k]\ni\al_k\mapsto \lambda_j(R_k^{ \al_k}),\quad j\in\N, 
\end{equation}
are continuous and nondecreasing. Moreover,
one has
\begin{equation}\label{lambda12+}
\lambda_1(R_k^{\al_k})=0,\quad\lambda_2(R_k^{ \al_k})=
\begin{dcases}
0,&\al_k=0,\\
(\pi / d_k )^2,&\al_k=d_k,
\end{dcases}
\end{equation}
and
\begin{equation}\label{lambda3+}
\lambda_3(R_k^{ \al_k})\ge \lambda_3(R_k^{0})= 
(\pi/d_k)^2.
\end{equation}

\noindent
{\it Step 2.}
Now we approach the main part of the proof, where the parameters will be properly adjusted. 
Let $0<\nu_1<\dots<\nu_m$ be as in the assumptions of the theorem and 
fix an open interval $\mathcal{I}$ such that 
\begin{gather}\label{Inu}
0<\inf \mathcal I<\nu_1\quad\text{and}\quad \nu_m<\sup\mathcal I. 
\end{gather}
Assume that the numbers $d_1,\dots,d_m$ 
satisfy
\begin{gather}\label{la<d}
\sup\mathcal I<\min_k\left(\pi/d_k\right)^2.
\end{gather}
Furthermore, let us choose a constant $\gamma>0$  such that the intervals 
$$ [\nu_k-\gamma,\nu_k+\gamma],\quad k=1,\dots,m,$$ 
are pairwise disjoint and  
\begin{gather}
\label{inI}
\bigcup_{k=1}^m [\nu_k-\gamma,\nu_k+\gamma]\subset\mathcal{I}
\end{gather}
holds. 
Next, we introduce the sets
\begin{equation}\label{interval}
L_k=\left\{\alpha_k\in[0,d_k]:\ \lambda_2(R_k^{ \al_k})\in [\nu_k-\gamma,\nu_k+\gamma]\right\}.
\end{equation}
Using the continuity and monotonicity of the function in \eqref{incre} and 
taking into account \eqref{lambda12+}, \eqref{Inu}-\eqref{inI}
we conclude that  each $ {L}_k$ is a nonempty compact interval.
We set 
$$\al_k^-=\min L_k,\quad \al_k^+=\max L_k,\quad\text{and}\quad \mathcal{D}=\prod_{k=1}^m[\al_k^-,\al_k^+].$$
It is clear from \eqref{lambda12+}, \eqref{Inu}, and \eqref{inI} that $\al_k^\pm>0$.
From now on we assume that
\begin{gather*}
\al=\{\al_1,\dots,\al_m\}\in \mathcal{D}.
\end{gather*}

As usual we denote 
the eigenvalues of the decoupled operator $A_{\Omega^{\al,0}}$ in \eqref{orto+} by $(\lambda_j(\Omega^{\al,0}))_{j\in\N}$, counted with multiplicities and ordered as a nondecreasing sequence.
It follows from $\lambda_1(R_k^{ \al_k})=0$ for $k=1,\dots,m$, that  
\begin{equation*}
\lambda_k(\Omega^{\al,0})=0\quad\text{for}\quad k=1,\dots,m.
\end{equation*}
Furthermore, \eqref{lambda3+}, \eqref{la<d}, \eqref{inI}, and \eqref{interval} imply that the $m+k$-th eigenvalues of the orthogonal sum $A_{\Omega^{\al,0}}$ coincides with the second eigenvalue 
$\lambda_{2}(R_k^{ \al_k})$ of the Neumann Laplacian $A_{R_k^{ \al_k}}$ as $k=1,\dots,m$:
\begin{gather}\label{lambda-all0}
\lambda_{m+k}(\Omega^{\al,0})=
\lambda_{2}(R_k^{ \al_k})\in \overline{B_\gamma(\nu_k)}\text{ for } k=1,\dots,m.
\end{gather}
Moreover, it is clear from \eqref{lambda3+} and \eqref{la<d} that
\begin{gather*}
\lambda_{k }(\Omega^{\al,0})>\sup\I \quad\text{for}\quad k= 2m+1,\,2m+2,\dots.
\end{gather*}

We introduce the functions $f_k^0:\mathcal{D}\to \R$ by 
\begin{gather}\label{f0}
f_k^0(\al_1,\al_2,\dots,\al_m)=\lambda_{k+m}(\Omega^{\al,0}),\quad k=1,\dots,m.
\end{gather}
It is important to note  that  due to \eqref{lambda-all0}
the value $\lambda_{k+m}(\Omega^{\al,0})$ of the function $f_k^0$ depends only on the $k$-th variable $\alpha_k$. Using this
and taking into account that the mapping \eqref{incre} is nondecreasing we get
\begin{multline}\label{Hempel0}
f_k^0(\al_1^+,\dots,\al_{k-1}^+,\al_k^-,\al_{k+1}^+,\dots,\al_m^+)\\
=\nu_k-\gamma
<\nu_k<\nu_k+\gamma
\\
=f_k^0(\al_1^-,\dots,\al_{k-1}^-,\al_k^+,\al_{k+1}^-,\dots,\al_m^-).
\end{multline}

\noindent
{\it Step 3.}
Let $(\lambda_j(\Omega^{\al,\eps}))_{j\in\N}$ be the eigenvalues 
of the Neumann Laplacian $A_{\Omega^{\al,\eps}}$ on $\Omega^{\al,\eps}$ counted with multiplicities and ordered
as a nondecreasing sequence. 

For $\eps\ge 0$ we introduce the  functions $f_k^\eps:\mathcal{D}\to \R$ by 
\begin{gather}\label{f}
f_k^\eps(\al_1,\al_2,\dots,\al_m)=\lambda_{k+m}(\Omega^{\al,\eps}),\quad k=1,\dots,m.
\end{gather}
Of course, for $\eps=0$ these functions coincide with the functions in \eqref{f0}.
Observe that, in contrast to $f_k^0$, for $\eps>0$ the values $\lambda_{k+m}(\Omega^{\al,\eps})$ of $f_k^\eps$ in general do not depend only on the $k$-th variable. 
It is important to note that Theorem~\ref{th-contin} and Remark~\ref{rem-contin} show that the functions
\begin{gather*}
\eps\mapsto \lambda_j(\Omega^{\al,\eps}),\quad j\in\N, 
\end{gather*}
are continuous for each fixed $\al$. Hence it follows together with \eqref{Hempel0} that 
\begin{equation}\label{Hempel}
\begin{split}
 f_k^\eps(\al_1^+,\dots,\al_{k-1}^+,\al_k^-,\al_{k+1}^+,\dots,\al_m^+)&<\nu_k\\
 &<f_k^\eps(\al_1^-,\dots,\al_{k-1}^-,\al_k^+,\al_{k+1}^-,\dots,\al_m^-)
 \end{split}
\end{equation}
for $\eps>0$ sufficiently small.
From now on we fix $\eps>0$ for which \eqref{Hempel} holds.

To proceed further we need  the following multidimensional version
of the intermediate value theorem, which was established in \cite[Lemma~3.5]{HKP97}.

\begin{lemma}[Hempel-Kriecherbauer-Plankensteiner, 1997]\label{lemma-hempel}
Let $a_k < b_k$, $k=1,\dots,m$, and $\mathcal{D}:=\prod_{k=1}^m[a_k, b_k]$. Assume
that $f:\mathcal{D}\to\R^m$ is continuous and that each component function $f_k$ is
nondecreasing in each of its arguments. Let us suppose that
$F_k^-<F_k^+$, $k=1,\dots,m$, where
\begin{gather*}
F_k^-=f_k(b_1,b_2,\dots,b_{k-1},a_k,b_{k+1},\dots,b_m),\\
F_k^+=f_k(a_1,a_2,\dots,a_{k-1},b_k,a_{k+1},\dots,b_m).
\end{gather*}
Then for any $F\in\prod_{k=1}^m[F_k^-,F_k^+]$
there exists  $x\in\mathcal{D}$ such that $f(x)=F$.
\end{lemma}\smallskip

We apply this lemma to the function 
$f^\eps=(f\e_1,\dots,f\e_m)$ defined by \eqref{f}.
By Theorem~\ref{th-contin} and Remark~\ref{rem-contin} $f^\eps:\mathcal{D}\to\R^m$ is continuous. Moreover, by the 
min-max principle each component $f^\eps_k$ of $f^\eps$ is nondecreasing in 
each of its arguments. 
Using this and \eqref{Hempel} we conclude that
$f^\eps$ satisfies all assumptions in Lemma~\ref{lemma-hempel} and 
hence there exists $\al\in\mathcal{D}$ such that
\begin{gather*}
f_k\e(  \al)=\nu_k.
\end{gather*}
With this choice of $\alpha=\{\alpha_1,\dots,\alpha_m\}$ and $d_1,\dots,d_m$ fixed as in the beginning of Step 2 (see \eqref{la<d}) it follows that
$\lambda_{m+k}(\Omega^{\alpha,\eps})=\nu_k$ for $k=1,\dots,m$. This completes  the proof of
Theorem~\ref{thCdV+}.  
\end{proof}

\section{Singular Schr\"odinger operators with $\delta'$-interactions\label{sec2}}

In this section we show that the methods used to control the spectrum of the Neumann Laplacian in the previous section 
can also be applied to singular Schr\"odinger operators describing the motion of quantum particles in 
potentials being supported at a discrete (finite or infinite) set of
points. These operators are often referred to as \textit{solvable models} in
quantum mechanics, since their mathematical and physical
quantities (e.g., their spectrum) can be determined explicitly. We refer to the monograph \cite{AGHH05} for an introduction to this topic. 
We also note that in the mathematical literature such operators are often called \textit{Schr\"odinger operators with point interactions}.  

The classical example of
a Schr\"odinger operator with \textit{$\delta$-interactions} is
the following formal expression
\begin{gather*}
 \ds-{\d^2 \over \d z^2} + \suml_{k\in \N}\al_k\delta_{z_k},
\end{gather*}
where $\delta_{z_k}$ are Dirac delta-functions supported at the points $z_k\in\R$ and $\al_k\in\R\cup\{\infty\}$. 
In the present paper we treat the closely related model of a Schr\"odinger operator with $\delta'$-interactions (or \textit{point dipole interactions})  
 defined by the formal expression
\begin{gather}\label{delta'}
-{\d^2\over \d z^2}+\suml_{k\in\N}\be_k\langle\cdot\,,\,\delta_{z_k}'\rangle\delta_{z_k}', 
\end{gather}
where $\delta_{z_k}'$ is the distributional derivative of the  delta-function supported at $z_k\in\R$,   $\langle\phi,\delta_{z_k}'\rangle$ denotes its action 
on the test function $\phi$, and $\be_k\in \R\cup\{\infty\}$.
The above formal expression can be realized  
as a self-adjoint operator in $\L$ with the action  $-{\d^2\over \d z^2}$ and
domain consisting of local $\H^2$-functions $u$ that satisfy
\begin{gather*}
u'(z_k -0) = u'(z_k  + 0),\quad u(z_k +0)-u(z_k -0)=\beta_k u'(z_k\pm 0 ) 
\end{gather*}
(the case $\beta=\infty$ stands for a decoupling with Neumann conditions at $z_k\pm 0$).
The existence of this model was pointed out by 
A.~Grossmann, R.~H{\o}egh-Krohn, M.~Mebkhout  in \cite{GHKM80}, the first rigorous mathematical treatment of $\delta'$-interactions is due to
F.~Gesztesy and H.~Holden in \cite{GH87}.
Among the numerous subsequent contributions we emphasize the more recent papers \cite{KM10b,KM14} by A.~Kostenko and M.M.~Malamud, in which also the 
more elaborate case $|z_k-z_{k-1}|\to 0$ as $|k|\to \infty$ was treated. 
In these papers self-adjointness,  lower semiboundedness and spectral properties of the underlying operators were studied in detail.

Our goal and strategy is similar to \cite{HSS91} in the context of the Neumann Laplacian: We wish to construct an 
operator of the form \eqref{delta'} with predefined essential spectrum; cf.~Theorem~\ref{th-HSS}.
At this point we present the main result of this section on a formal level
without giving a precise definition of  the underlying operator.  
This will be done during its proof; cf.~Theorem~\ref{th-BK+} for a more precise formulation of Theorem~\ref{th-BK}.

\begin{theorem} \label{th-BK}
Let $\S\subset [0,\infty)$ be an arbitrary closed set such that  $0\in \S$. Then there exists a bounded interval $(a,b)\subset\mathbb{R}$, a sequence of points 
$(z_k)_{k\in\N}$ in $(a,b)$, and a
sequence of positive numbers $(\beta_k)_{k\in\N}$ such that the operator $A_\beta $ in $\L(a,b)$ defined by  formal expression \eqref{delta'} satisfies 
\begin{gather*}
\sigma_\ess(A_\beta)=\S.
\end{gather*}
\end{theorem}

\begin{proof}
For the construction of the self-adjoint 
Schr\"odinger operator with $\delta'$-interactions we use a similar idea as in the construction of the
rooms-and-passages domains in the previous section. 
Here we split the sequence of points  $(z_k)_{k\in\N}$ in \eqref{delta'} in two interlacing
subsequences $(x_k)_{k\in\N}$ and $(y_k)_{k\in\N}$, where the point $y_k$ is in the middle of $(x_{k-1},x_k)$,
and instead of $\beta_k$ we denote the interaction strengths at the points $x_k$ by $p_k$ and at the points $y_k$ by $q_k$. 
Instead of $A_\beta$ we shall write $A_{p,q}$ for the corresponding Schr\"{o}dinger operator with $\delta'$-interactions, see Step 3 of the proof.
The intervals $(x_{k-1},x_k)$ will  play  the role of the rooms, the interactions at the points $x_k$ will play  the role of 
the passages, and the interactions at the points $y_k$ will play the role of the additional walls inside the rooms.   
\medskip

As in the proof of Theorem~\ref{th-HSS} we fix a sequence $(s_k)_{k\in \N}$
such that 
\begin{gather}\label{s-assumpt2}
s_k> 0\text{ and }\acc((s_k)_{k\in \N})=
\begin{cases}
\S\setminus\{0\},&0\text{ is an isolated point of }\S,\\
\S,&\text{otherwise},
\end{cases}
\end{gather}
and for each $k\in\N$ we choose numbers $d_k>0$ that satisfy
\begin{gather}\label{dk}
s_k < (\pi/ d_k)^2.
\end{gather} 
Moreover, we can assume that $d_k$ are chosen sufficiently small, so that 
\begin{gather}\label{dk+}
\suml_{k\in \N} d_k <\infty
\end{gather}
and hence
\begin{gather}\label{dk-infty}
\lim_{k\to\infty} d_k=0
\end{gather} 
holds.
Finally, we set 
\begin{gather*}
x_0=0,\quad x_k=x_{k-1}+d_k,\quad 
y_k={x_{k-1}+x_k\over 2},\quad \I_k=(x_{k-1},x_k),\,\, k\in\N,
\end{gather*}
and we consider the interval $(a,b)$, where
\begin{gather*}
a=x_0=0\quad\text{and}\quad b=\suml_{k\in \N} d_k.
\end{gather*}

The proof consists of four steps. In the first step we discuss the spectral properties of the Schr\"odinger operator $\Ak_{q_k,\I_k}$ 
on the interval $\I_k$ with a $\delta'$-interaction of strength $q_k>0$ at $y_k$
and Neumann boundary conditions at the endpoints of $\I_k$.
In the second step we consider the direct sum of these operators:
$$A_{\infty,q}=\bigoplus_{k\in\N}\Ak_{q_k,\I_k}.$$ 
Note that the Neumann conditions at $x_k\pm 0 $ can be regarded as $\delta'$-interaction with \textit{infinite} strength.
Thus $A_{\infty,q}$ corresponds to the Schr\"odinger operator on $(a,b)$
with $\delta'$-interactions of strengths $q_k$ at the points $y_k$ and $\delta'$-interactions of strengths $\infty$ at the points $x_k$.
The interaction strengths $q_k$ will be adjusted in such a way that 
the essential spectrum of $A_{\infty,q}$ coincides with $\S$.
In fact, $\sigma_{\ess}(A_{\infty,q})$ is the union of the point $0$ and all 
accumulation points of a sequence formed by the \textit{second} eigenvalues
of $\Ak_{q_k,\I_k}$.
In the third step we perturb the decoupled operator $A_{\infty,q}$ linking the intervals $\I_{k+1}$ and $\I_k$
by a $\delta'$-interaction of a sufficiently large strength $p_k>0$ for all $k\in\N$;
the corresponding operator is denoted by $A_{p,q}$. We will prove in the last step that 
the essential spectra of $A_{p,q}$ and $A_{\infty,q}$ coincide if the interaction strengths
$p_k$ tend to $\infty$ for $k\to\infty$ sufficiently fast.
\medskip

\noindent\textit{Step~1.}
Let $q_k\in (0,\infty]$ and let $\ak_{q_k,\I_k}$ be the sesquilinear form in $\L(\I_k)$
defined by
\begin{equation*}
\begin{split}
\ak_{q_k,\I_k}[\u,\vv]&=\int_{\I_k}\u'\cdot \overline{\vv'}\,\d x \\
&\qquad + \frac{1}{q_k} 
\left(\u(y_k+0)-\u(y_k-0)\right)\overline{\left(\vv(y_k+0)-\vv(y_k-0)\right)},\\
\dom(\ak_{q_k,\I_k})&=\H^1(\I_k\setminus\{y_k\});
\end{split}
\end{equation*}
for $q_k=\infty$ we use the convention $\infty^{-1}=0$.
The form $\ak_{q_k,\I_k}$ is densely defined, nonnegative, and closed in $\L(\I_k)$. Hence 
by the first representation theorem there is a unique nonnegative 
self-adjoint operator $\Ak_{q_k,\I_k}$ in $\L(\I_k)$ such that 
$\dom(\Ak_{q_k,\I_k})\subset\dom(\ak_{q_k,\I_k})$ and
\begin{equation*}
(\Ak_{q_k,\I_k} \u, \vv)_{\L(\I_k)}=\ak_{q_k,\I_k}[\u,\vv],\quad \u\in \dom(\Ak_{q_k,\I_k}),\,\,\vv\in\dom(\ak_{q_k,\I_k}).
\end{equation*}
Integration by parts shows that
$\dom(\Ak_{q_k,\I_k})$ consists of all those functions $\u\in \H^2(\I_k\setminus\{y_k\})$ that satisfy
the $\delta'$-jump condition
\begin{equation*}
\u'(y_k -0) = \u'(y_k  + 0)=\frac{1}{q_k}\left( \u(y_k +0)-\u(y_k -0)\right)
\end{equation*}
at the point $y_k$ and Neumann boundary conditions
\begin{equation*}
\u'(x_{k-1}) = \u'(x_k)=0
\end{equation*}
at the endpoints of $\I_k$. Furthermore,
the action of $\Ak_{q_k,\I_k}$ is given by
$$
(\Ak_{q_k,\I_k}\u )\!\restriction_{(x_{k-1},y_k)}= -\left( \u\!\restriction_{(x_{k-1},y_k)}\right)'',\quad
(\Ak_{q_k,\I_k}\u )\!\restriction_{(y_k,x_{k})}= -\left( \u\!\restriction_{(y_k,x_k)}\right)''.
$$ 
The spectrum of the self-adjoint operator $\Ak_{q_k,\I_k}$ is purely discrete. We use the notation
$\left\{\lambda_j(\Ak_{q_k,\I_k})\right\}_{j\in\N}$ for the corresponding eigenvalues counted with multiplicities and ordered
as a nondecreasing sequence. Some properties of these eigenvalues are collected in the next lemma. Here we will also make use of the Neumann Laplacian
on $\I_k$, defined as usual via the form 
\begin{equation}\label{neum22}
\ak_{0,\I_k}[\u,\vv]=(\u',\vv')_{\L(\I_k)},\quad \dom(\ak_{0,\I_k})=\H^1(\I_k),
\end{equation}
and we shall denote this operator by $\Ak_{0,\I_k}$. To avoid possible confusion we emphasize that the form domain $\H^1(\I_k)$
of the Neumann Laplacian $\Ak_{0,\I_k}$ is smaller than the form domain $\H^1(\I_k\setminus\{y_k\})$ of the 
operators $\Ak_{q_k,\I_k}$ with $q_k\in (0,\infty]$. Furthermore, we mention already here that the self-adjoint operator $\Ak_{\infty,\I_k}$ is the direct sum of the 
Neumann Laplacians on the intervals $(x_{k-1},y_k)$ and $(y_k,x_k)$.

\begin{lemma}\label{lemma:EVprop}
For each $j\in\N$ 
\begin{equation}\label{cont-monot}
\begin{array}{l}
\text{the function }(0,\infty]\ni q_k\mapsto \lambda_j(\Ak_{ q_k,\I_k})\text{ is}\\
\text{monotonically decreasing and continuous,}
\end{array}
\end{equation}
and one has
\begin{equation}\label{0} 
\liml_{q_k\to +0}\lambda_j(\Ak_{ q_k,\I_k})=
\lambda_j(\Ak_{0,\I_k}).
\end{equation}
\end{lemma}

\begin{proof}[Proof of Lemma~\ref{lemma:EVprop}]
The monotonicity of the function \eqref{cont-monot} 
follows from the min-max principle and the monotonicity of the function
$q_k\mapsto \ak_{q_k,\I_k}[\u,\u]$
for each fixed $\u\in\H^1(\I_k)$.
To prove the continuity of the function \eqref{cont-monot} consider $q_k,\widehat q_k\in (0,\infty]$, 
$\mathbf{f},\mathbf{g}\in\L(\I_k)$, and set $\u=(\Ak_{q_k,\I_k}+\Id)^{-1}\mathbf{f}$ and 
$\vv=(\Ak_{\widehat q_k,\I_k}+\Id)^{-1}\mathbf{g}$. Then we have
\begin{equation}\label{res:diff}
\begin{split}
 &\bigl((\Ak_{q_k,\I_k}+\Id)^{-1}\mathbf{f}-(\Ak_{\widehat q_k,\I_k}+\Id)^{-1}\mathbf{f},\mathbf{g}\bigr)_{\L(\I_k)}\\
 &\quad=\bigl(\u,(\Ak_{\widehat q_k,\I_k}+\Id)\vv\bigr)_{\L(\I_k)}-\bigl((\Ak_{q_k,\I_k}+\Id)\u, \vv\bigr)_{\L(\I_k)}\\
 &\quad=\a_{\widehat q_k,\I_k}[\u,\vv]-\a_{q_k,\I_k}[\u,\vv]\\
 &\quad=\left(\frac{1}{\widehat q_k}-\frac{1}{q_k}\right)
\left(\u(y_k+0)-\u(y_k-0)\right)\overline{\left(\vv(y_k+0)-\vv(y_k-0)\right)}
\end{split}
 \end{equation}
With $\I_k^+=(y_k,x_k)$ and $\I_k^-=(x_{k-1},y_k)$ one has 
the standard trace estimate (see, e.g. \cite[Lemma 1.3.8]{BKbook}) 
\begin{gather*}
|\u(y_k\pm 0)|^2\leq 
\frac{d_k}{2}\|\u'\|^2_{\L(\I^\pm_k)}+
\frac{4}{d_k}\|\u\|^2_{\L(\I^\pm_k)},\quad \u\in\H^1(\I^\pm_k),
\end{gather*}
and with $C_k=\max\{d_k,8 d_k^{-1}\}$ we estimate
\begin{equation}\label{okpr}
 \begin{split}
  \bigl\vert\u(y_k+0)-\u(y_k-0) \bigr\vert^2 &\leq 2  \vert\u(y_k+0)\vert ^2+  2  \vert\u(y_k-0)\vert ^2\\
 &\leq  d_k\|\u'\|^2_{\L(\I_k\setminus\{y_k\})}+
8d^{-1}_k\|\u\|^2_{\L(\I_k)}\\
&\leq C_k \bigl(\ak_{q_k,\I_k}[\u,\u]+\|\u\|^2_{\L(\I_k)}\bigr)\\
&= C_k \bigl((\Ak_{q_k,\I_k}+\Id)\u, \u\bigr)_{\L(\I_k)}\\
&= C_k \bigl(\mathbf{f},(\Ak_{q_k,\I_k}+\Id)^{-1}\mathbf{f}\bigr)_{\L(\I_k)}\\
& \leq C_k \|\mathbf{f}\|_{\L(\I_k)} \|(\Ak_{q_k,\I_k}+\Id)^{-1}\mathbf{f}\|_{\L(\I_k)}\\
& \leq C_k \|\mathbf{f}\|^2_{\L(\I_k)},
 \end{split}
\end{equation}
where we have used $q_k>0$ in the third estimate. 
In the same way we get $\vert\vv(y_k+0)-\vv(y_k-0)\vert^2\leq C_k \|\mathbf{g}\|^2_{\L(\I_k)}$. Hence \eqref{res:diff}
leads to the estimate
\begin{equation*}
\begin{split}
 \bigl|\bigl((\Ak_{q_k,\I_k}+\Id)^{-1}\mathbf{f}&-(\Ak_{\widehat q_k,\I_k}+\Id)^{-1}\mathbf{f},\mathbf{g}\bigr)_{\L(\I_k)}\bigr| \\
 &\qquad\leq C_k 
 \left|\frac{1}{\widehat q_k}-\frac{1}{q_k}\right| \|\mathbf{f}\|_{\L(\I_k)} \|\mathbf{g}\|_{\L(\I_k)},
\end{split}
 \end{equation*}
and from this we conclude 
\begin{gather}\label{nrc}
\|(\Ak_{q_k,\I_k}+\Id)^{-1} -(\Ak_{\widehat q_k,\I_k}+\Id)^{-1}\|\to 0
\text{ as }\widehat q_k\to q_k.
\end{gather}
It is well-known (see, e.g., \cite[Corollary~A.15]{P06}) that the  norm-resolvent convergence in \eqref{nrc} implies the convergence of the eigenvalues, namely
for each $j\in\N$ we obtain 
$$\lambda_j(\Ak_{\widehat q_k,\I_k})\to  \lambda_j(\Ak_{ q_k,\I_k})
\quad\text{as}\quad\widehat q_k\to q_k,$$
and hence the function in \eqref{cont-monot} is continuous.

It remains to prove \eqref{0}. For this we will use Theorem~\ref{th-s+} from Appendix~\ref{appb}.
Note first that
the set 
$$\left\{\u\in \H^1(\I_k\setminus\{y_k\})=\dom(\ak_{ q_k,\I_k}):\ \sup_{q_k>0}\ak_{ q_k,\I_k}[\u,\u]<\infty\right\}$$ 
coincides with the form domain $\dom(\ak_{0,\I_k})=\H^1(\I_k)$ of the Neumann Laplacian in \eqref{neum22}. Moreover, 
for each $\u,\vv$ from this set one has $$\lim_{q_k\to 0}\ak_{q_k,\I_k}[\u,\vv]=\ak_{0,\I_k}[\u,\vv].$$
Since the spectra of the operators $\Ak_{q_k,\I_k}$ and $\Ak_{0,\I_k}$ are purely discrete 
Theorem~\ref{th-s+} shows \eqref{0}. 
\end{proof}

Now we return to the spectral properties of the self-adjoint operators $\Ak_{ q_k,\I_k}$. In particular, the eigenvalues 
of the Neumann Laplacian $\Ak_{0,\I_k}$ on $\I_k$ and the direct sum of the Neumann Laplacians $\Ak_{\infty,\I_k}$ 
on $(x_{k-1},y_k)$ and $(y_k,x_k)$ can be easily calculated. For our purposes it suffices to note that
\begin{gather}\label{1}
\lambda_1(\Ak_{0,\I_k})=\lambda_1(\Ak_{\infty,\I_k})=0,\\\label{2}
\lambda_2(\Ak_{0,\I_k})=\left({\pi/ d_k}\right)^2,\quad \lambda_2(\Ak_{\infty,\I_k})=0,\\
\label{3}
\lambda_3(\Ak_{\infty,\I_k})=\left({2\pi/ d_k}\right)^2.
\end{gather}
It follows from \eqref{cont-monot}, \eqref{1}, and \eqref{3} that 
for any $q_k\in(0,\infty]$ we have
\begin{gather}\label{13+}
\lambda_1(\Ak_{q_k,\I_k})=0,\quad 
\lambda_3(\Ak_{q_k,\I_k})\geq\left({2\pi/ d_k}\right)^2.
\end{gather}
Also, using \eqref{cont-monot}, \eqref{0}, \eqref{2}  
and taking into account that $0<s_k < (\pi/ d_k)^2$ from \eqref{dk} we conclude that there exists $q_k> 0$ such that
\begin{gather}
\label{2+}
\lambda_2(\Ak_{q_k,\I_k})=s_k,\quad k\in\N.
\end{gather}
From now on we fix $q_k>0$ for which \eqref{2+} holds.
\medskip

\noindent\textit{Step 2}.
Now we consider the direct sum  
\begin{equation}\label{ai8}
A_{\infty,q}=\bigoplus_{k\in\N} \Ak_{q_k,\I_k}
\end{equation}
of the nonnegative self-adjoint operators $\Ak_{q_k,\I_k}$ in the space
$$
\L(a,b)=\bigoplus_{k=1}^\infty \L(\I_k).
$$
In a more explicit form $A_{\infty,q}$ is given by
\begin{equation*}
\begin{split}
(A_{\infty,q}u)\!\restriction_{\I_k}&=\Ak_{q_k,\I_k}\u_k,\\
\dom(A_{\infty,q})&=
\biggl\{
u\in\L(a,b):\ \u_k\in\dom(\Ak_{q_k,\I_k}),\\
&\qquad\qquad\qquad\qquad
\suml_{k\in\N}\|\Ak_{q_k,\I_k}\u_k\|^2_{\L(\I_k)}<\infty
\biggr\},
\end{split}
\end{equation*}
where $\u_k:=u\!\restriction_{\I_k}$ stands for 
the restriction of the function $u$ onto the interval $\I_k$. Note that the corresponding sesquilinear form $\a_{\infty,q}$ associated 
 with $A_{\infty,q}$ is 
\begin{equation*}
\begin{split}
\ak_{\infty,q}[u,v]&=\suml_{k\in\N}\ak_{q_k,\I_k}[\u_k,\vv_k],\\
\dom(\a_{\infty,q})&=
\left\{
u\in\L(a,b):\ \u_k\in\dom(\ak_{q_k,\I_k}),\
\suml_{k\in\N}\ak_{q_k,\I_k}[\u_k,\u_k] <\infty
\right\}.
\end{split}
\end{equation*}
It is clear that the operator $A_{\infty,q}$ in \eqref{ai8} is self-adjoint and nonnegative in $\L(a,b)$. 
Furthermore, it is not difficult to check that
\begin{gather*} 
\sigma_\ess(A_{\infty,q})
=
\acc\big((\lambda_j(\Ak_{q_k,\I_k}))_{j,k\in\N}\big)
\end{gather*}
holds. Taking into account
that $0\in \S$ and using \eqref{s-assumpt2}, \eqref{dk-infty}, \eqref{13+}, \eqref{2+}, we arrive at
\begin{gather*}
\sigma_\ess(A_{\infty,q})
=  
\{0\}\cup\acc\big((s_k)_{k\in\N}\big)
= \S.
\end{gather*} 

\noindent
\textit{Step~3.} 
In this step we perturb the decoupled operator $A_{\infty,q}$ linking the intervals $\I_{k+1}$ and $\I_k$
by a $\delta'$-interaction of sufficiently large strength $p_k>0$ for all $k\in\N$.
The corresponding self-adjoint operator will be denoted by $A_{p,q}$. 
More precisely, for $p_k>0$, $k\in\N$, we consider the sesquilinear form
$\a_{p,q}$
\begin{equation*}
\begin{split}
\a_{p,q}[u,v]&=\suml_{k\in\N}\ak_{q_k,\I_k}[\u_k,\vv_k]\\ 
&\quad +\suml_{k\in\N} \frac{1}{p_k} 
\left(u(x_k+0)-u(x_k-0)\right)
\overline{\left(v(x_k+0)-v(x_k-0)\right)},\\
\dom(\a_{p,q})&=\left\{u\in \L(a,b):\ \u_k\in\dom(\ak_{q_k,\I_k}),\, \a_{p,q}[u,u]<\infty\right\},
\end{split}
\end{equation*}
in $\L(a,b)$. This form is nonnegative and densely defined in $\L(a,b)$.
Moreover, the form is closed by \cite[Lemma~2.6]{KM14} and the corresponding nonnegative self-adjoint operator $A_{p,q}$ 
is given by 
\begin{equation*}
\begin{split}
(A_{p,q}u)\!\restriction_{(a,b)\setminus\mathcal{Z}}&=-(u\!\restriction_{(a,b)\setminus\mathcal{Z}})'',\\
\dom(A_{p,q})&=\bigg\{u\in\H^2((a,b)\setminus\mathcal{Z}):\ u'(a)=0,\\
&\qquad u'(y_k +0) = u'(y_k  - 0)=\frac{1}{q_k}\left( u(y_k +0)-u(y_k -0)\right),\\
&\qquad u'(x_k +0) = u'(x_k  - 0)=\frac{1}{p_k}\left( u(x_k +0)-u(x_k -0)\right)\bigg\},
\end{split}
\end{equation*}
where $\mathcal{Z}=\left\{x_k:k\in\N\right\}\cup\left\{y_k: k\in\N\right\}$; cf.~\cite[Lemma~2.6 and Proposition~2.1]{KM14}.
\medskip 
Now consider  
$$
\rho_k:=\max\left\{{1\over p_k d_k},\ {1\over p_kd_{k+1}}\right\},\quad k\in\N,
$$
and assume that 
\begin{gather}\label{rho}
 \rho_k\to 0\text{ as }k\to\infty.
 \end{gather} 

\noindent
\textit{Step~4.}  
In this step we verify 
\begin{equation}\label{kmt}
\sigma_\ess(A_{p,q})=\sigma_\ess(A_{\infty,q}).
\end{equation}
by showing that 
the  difference of resolvents
$$T_{p,q}:=(A_{p,q} + \Id)^{-1}-(A_{\infty,q} + \Id)^{-1}$$
is a compact operator. 
Then \eqref{kmt} is an immediate consequence of
the Weyl theorem, see, e.g.  \cite[Theorem XIII.14]{RS78}. {  We remark that in a similar situation a related perturbation result 
and the invariance of the essential spectrum was shown in \cite[Theorem~1.3]{KM14}}.

In the following let $\kappa_n=\sup_{k\in [n,\infty)\,\cap\,\N}\rho_k$. Then 
it follows from \eqref{rho} that
\begin{gather}
\label{kappa}
\kappa_n<\infty\text{ for each }n\in\N\text{\quad and\quad }
\kappa_n\to 0\text{ as }n\to\infty. 
\end{gather}

In a first step we claim that
\begin{gather}\label{domdom+}
\dom(\a_{\infty,q})=\dom(\a_{p,q}).
\end{gather}
In fact, the inclusion $\dom(\a_{p,q})\subset\dom(\a_{\infty,q})$
follows directly from the definition of the above form domains.
To prove the reverse inclusion
 we have to show that
\begin{gather}\label{sum:infty}
\suml_{k\in\N}\frac{1}{p_k}
\left|u(x_k+0)-u(x_k-0)\right|^2<\infty
\end{gather}
for 
$u\in\dom(\a_{\infty,q})$.
Using the standard trace estimates (see, e.g. \cite[Lemma~1.3.8]{BKbook}) 
\begin{equation*} 
 \begin{split}
|u(x_k+0)|^2&\leq d_{k+1}\|u'\|^2_{\L(\I_{k+1})}+ \frac{2}{d_{k+1}}\|u\|^2_{\L(\I_{k+1})},\\
|u(x_k-0)|^2&\leq d_{k}\|u'\|^2_{\L(\I_k)}+ \frac{2}{d_k}\|u\|^2_{\L(\I_k)},
 \end{split}
\end{equation*}
and taking into account that $\sup_{k\in\N}d_k<b-a$ and $q_k>0$  we obtain
\begin{equation}\label{form:encl}
\begin{split}
&\suml_{k\in\N}\frac{1}{p_k}
\left|u(x_k+0)-u(x_k-0)\right|^2 \\
&\quad \leq 2\suml_{k\in\N}\frac{1}{p_k}|u(x_k+0)|^2+2\suml_{k\in\N}\frac{1}{p_k}|u(x_k-0)|^2\\
&\quad 
\leq 2\suml_{k\in\N}\frac{1}{p_k}\left(
d_{k+1}\|u'\|^2_{\L(\I_{k+1})} + \frac{2}{d_{k+1}}\|u\|^2_{\L(\I_{k+1})} \right)\\ 
&\qquad\qquad + 2\suml_{k\in\N}\frac{1}{p_k}\left(
d_k\|u'\|^2_{\L(\I_k)} + \frac{2}{d_k}\|u\|^2_{\L(\I_k)} \right)\\ 
&\quad\leq 2\kappa_1  \suml_{k\in\N} d_{k+1}^2\|u'\|^2_{\L(\I_{k+1})} +  4\kappa_1  \suml_{k\in\N} \|u\|^2_{\L(\I_{k+1})} \\
&\qquad\qquad +2\kappa_1  \suml_{k\in\N} d_{k}^2\|u'\|^2_{\L(\I_{k})} +  4\kappa_1  \suml_{k\in\N} \|u\|^2_{\L(\I_{k})} \\
&\quad \le  
4\kappa_1 (b-a)^2  \|u'\|^2_{\L(a,b)}+8\kappa_1  \|u\|^2_{\L(a,b)}\\
&\quad\leq
4\kappa_1 (b-a)^2\a_{\infty,q}[u,u]+8\kappa_1  \|u\|^2_{\L(a,b)},
\end{split}
\end{equation}
and thus \eqref{sum:infty} holds. We have shown \eqref{domdom+}.

Now let $f,g\in\L(a,b)$ be arbitrary and consider the functions 
\begin{equation*}
\begin{split}
u&=(A_{p,q}+\Id)^{-1}f\in \dom(A_{p,q})\subset\dom(\a_{p,q}),\\
v&=(A_{\infty,q}+\Id)^{-1}g\in \dom(A_{\infty,q})\subset\dom(\a_{\infty,q}).
\end{split}
\end{equation*}
Using \eqref{domdom+}
and the fact that $(A_{\infty,q}+\Id)^{-1}$ is a self-adjoint operator  we get
\begin{equation}\label{res-dif}
\begin{split}
\left(T_{p,q} f,g\right)_{\L(a,b)}&=
\left((A_{p,q}+\Id)^{-1}f - (A_{\infty,q}+\Id)^{-1} f,g\right)_{\L(a,b)}\\
&=
\left(u,(A_{\infty,q}+\Id)v\right)_{\L(a,b)}-\left((A_{p,q}+\Id)u, v\right)_{\L(a,b)}\\
&=\a_{\infty,q}[u,v] - \a_{p,q}[u,v]\\
&=-\suml_{k\in\N}\frac{1}{p_k} (u(x_k+0)-u(x_k-0))\overline{(v(x_k+0)-v(x_k-0))}.
\end{split}
\end{equation}
Next we introduce the operators $\Gamma_{p},\Gamma_{\infty}:\L(a,b)\to l^2(\N)$  defined by
\begin{equation*}
\begin{split}
(\Gamma_{p} f)_k&:={((A_{p,q}+\Id)^{-1}f)(x_k+0)-((A_{p,q}+\Id)^{-1}f)(x_k-0)\over \sqrt{p_k}},\\
(\Gamma_{\infty} g)_k&:={((A_{\infty,q}+\Id)^{-1}g)(x_k+0)-((A_{\infty,q}+\Id)^{-1}g)(x_k-0)\over \sqrt{p_k}},
\end{split}
\end{equation*}
on their natural domains
\begin{equation*}
\begin{split}
\dom(\Gamma_{p})&=\left\{f\in\L(a,b):\ \Gamma_{p} f\in l^2(\N)\right\},\\
\dom(\Gamma_{\infty})&=\left\{g\in\L(a,b):\ \Gamma_{\infty} g\in l^2(\N)\right\}.
\end{split}
\end{equation*}
Note that $\dom(\Gamma_{p})$ coincides with the whole $\L(a,b)$; this follows immediately from the $\ran(A_{p,q}+\Id)^{-1}\subset \dom(\a_{p,q})$.
Let us prove that the operator
 $\Gamma_p$ is  compact.
For this purpose we introduce the finite rank operators 
$$
\Gamma_p^n:\L(a,b)\to l^2(\N),\qquad (\Gamma_p^n f)_k = \begin{cases} (\Gamma_p f)_k, & k\le n,\\ 
 0, & k> n. \end{cases}
$$
Let $f\in\L(a,b)=\dom(\Gamma_p)$ and $u=(A_{p,q}+\Id)^{-1}f$.
Using the same arguments as in the proof of \eqref{form:encl} and \eqref{okpr} we
obtain 
\begin{equation}\label{f-est}
\begin{split}
\|\Gamma_p^n f - \Gamma_p f\|^2_{l^2(\N)}&=
\suml_{k:\,k>n}\frac{1}{p_k}|u(x_k+0)-u(x_k-0)|^2\\ 
&\leq 4\kappa_{n+1} (b-a)^2 \left(\a_{p,q}[u,u] +  \|u\|^2_{\L(a,b)}\right)+8\kappa_{n+1} \|u\|^2_{\L(a,b)} \\
&=4\kappa_{n+1} (b-a)^2(f,u)_{\L(a,b)} +8\kappa_{n+1} \|u\|^2_{\L(a,b)}\\
&\leq (4\kappa_{n+1} (b-a)^2  +8\kappa_{n+1}) \|f\|^2_{\L(a,b)}
\end{split}
\end{equation}
and hence it follows from \eqref{kappa} that $\|\Gamma_p^n  - \Gamma_p\|_{l^2(\N)} \to 0$ as $n\to\infty$. Since $\Gamma_p^n$ are finite rank operators
we conclude that the operator $\Gamma_p$ is compact. 
Furthermore, it is easy to see that  $\Gamma_\infty$ is a bounded operator defined on $\L(a,b)$.
Indeed, for $g\in\L(a,b)$ and $v=(A_{\infty,q}+\Id)^{-1}g$ one verifies in the same way as in
\eqref{form:encl} and \eqref{f-est} that
\begin{equation*}
 \begin{split}
\|\Gamma_\infty g\|^2_{l^2(\N)}&=
\suml_{k\in\N}(p_k)^{-1}|v(x_k+0)-v(x_k-0)|^2\\
&\le 
(4\kappa_1 (b-a)^2  +8\kappa_1) \|g\|^2_{\L(a,b)}.
\end{split}
\end{equation*}
Now \eqref{res-dif} can be rewritten in the form 
$$(T_{p,q}f,g)_{\L(a,b)}=-(\Gamma_p f,\Gamma_\infty g)_{l^2(\N)},\qquad f,g\in\L(a,b),$$ 
and hence we have
$T_{p,q} =-(\Gamma_\infty)^*\Gamma_p.$ Since $\Gamma_p$ is compact and 
$\Gamma_\infty$ is bounded (thus $(\Gamma_\infty)^*$ is also bounded) we conclude that 
$T_{p,q}$ is compact. 
\end{proof}

For the convenience of the reader we now formulate Theorem~\ref{th-BK} in a more precise form.

\begin{theorem}\label{th-BK+}
Let $\S\subset [0,\infty)$ be an arbitrary closed set such that  $0\in \S$ and choose
the sequences $(s_k)_{k\in \N}$ and $(d_k)_{k\in \N}$ as in \eqref{s-assumpt2}--\eqref{dk+}. 
Let
$(q_k)_{k\in \N}$ and $(p_k)_{k\in \N}$ be sequences such that \eqref{2+} and \eqref{rho} hold. Then the self-adjoint 
Schr\"odinger operator $A_{p,q}$ with $\delta'$-interactions of strengths $(q_k)_{k\in \N}$ and $(p_k)_{k\in \N}$ at the points $(y_k)_{k\in \N}$ and $(x_k)_{k\in \N}$,
respectively, satisfies
$$\sigma_\ess(A_{p,q})=\S.$$
\end{theorem}

At the end of this section we note that there exist many other methods for the construction of Schr\"odinger operators
with predefined spectral properties. For example, F.~Gesztesy, W.~Karwowski, and Z.~Zhao  constructed in \cite{GKZ92a,GKZ92b} a smooth
potential $V$ (which is a limit of suitably chosen  
$N$-soliton solutions of the Korteweg-de Vries equation as $N\to\infty$) such that the Schr\"odinger operator $H = -{\d^2\over \d x^2} + V$ has purely
absolutely continuous spectrum $\R_+$ and a prescribed sequence of points in $\R_-$ is contained in the set
of eigenvalues of $H$.

\section{Essential spectra of self-adjoint extensions of symmetric operators\label{sec3}}

The aim of this slightly more abstract section is to discuss some possible spectral 
properties of self-adjoint extensions of a given symmetric operator in a separable Hilbert space. In a similar context the existence of self-adjoint extensions 
with prescribed point spectrum, absolutely continuous spectrum, and singular continuous spectrum in spectral gaps of a fixed underlying symmetric operator 
was discussed in \cite{ABMN05,ABN98,B04,BMN06,BN95,BN96,BNW93}, see also \cite{M79} for a related result on prescribed eigenvalue asymptotics {\color{black}or, e.g., the earlier contributions \cite{AI71,G69,G70,I71}}. Our main observation here is the fact that for a symmetric operator with infinite defect numbers 
one can construct self-adjoint extensions with prescribed essential spectrum; cf. Theorem~\ref{essit} below.

In the following let $\cH$ be a separable (infinite dimensional) complex Hilbert space with scalar product $(\cdot,\cdot)$. Recall that a linear operator $S$ in $\cH$
is said to be {\it symmetric} if
\begin{equation*}
(Sf,g)=(f,Sg),\qquad f,g\in\dom (S).
\end{equation*}
We point out that a symmetric operator is in general not self-adjoint. More precisely, if the domain $\dom (S)$ of $S$ is dense in $\cH$ then 
the adjoint $S^*$ of the operator $S$ is given by
\begin{equation*}
\begin{split}
S^*h&=k,\\ 
\dom (S^*)&=\bigl\{h\in\cH:\exists\,k\in\cH\text{ such that } (Sf,h)=(f,k)\text{ for all }f\in\dom (S)\bigr\},
\end{split}
\end{equation*}
and the fact that $S$ is symmetric is equivalent to the inclusion $S\subset S^*$ in the sense that $\dom (S)\subset\dom (S^*)$ and $S^*f=Sf$ for all
$f\in\dom (S)$. However, this is obviously a weaker property than the more natural physical property of {\it self-adjointness}, that is, $S=S^*$.
A symmetric operator is not necessarily closed (although closable) and the spectrum of a symmetric operator which is not self-adjoint 
typically covers the whole complex plane (or at least the upper or lower complex halfplane). We also point out that the closure $\overline S$ 
of a symmetric operator $S$
is not necessarily self-adjoint; if this is the case such an operator is called {\it essentially self-adjoint}, that is, $\overline S=S^*$ -- however,
we shall not deal with essential self-adjoint operators here.
We emphasize that from a spectral theoretic point of view
a symmetric operator (or an essentially self-adjoint operator) which is not self-adjoint is not suitable as an observable in the description of a 
physical quantum system. 

It is an important issue to understand in which situations a symmetric operator admits self-adjoint extensions and how these self-adjoint extensions
can be described. 
These questions were already discussed in the classical contribution \cite{N30} by J. von Neumann. 
For completeness we recall that a self-adjoint operator $A$ in $\cH$ is an extension of a densely defined symmetric operator $S$ 
if $S\subset A$; since $A$ is self-adjoint this is equivalent to $A\subset S^*$. 
We start by recalling the so-called first von
Neumann formula in the next theorem.

\begin{theorem}
	Let $S$ be a densely defined closed symmetric operator in $\cH$. Then the domain of the adjoint operator $S^*$ admits the direct sum decomposition
	\begin{equation}\label{neu1}
	\dom (S^*)=\dom (S)\,\dot+\,\ker(S^*-i)\,\dot+\,\ker(S^*+i).
	\end{equation}
\end{theorem}

Note that $S^*f_i=if_i$ for all $f_i\in\ker(S^*-i)$ and similarly $S^*f_{-i}=-if_{-i}$ for all $f_{-i}\in\ker(S^*+i)$. The spaces $\ker(S^*-i)$ and 
$\ker(S^*+i)$ are usually called {\it defect spaces} of $S$ and their dimensions are the {\it deficiency indices} of $S$. 
It will turn out that the deficiency indices and isometric operators in between the defect spaces are particularly 
important in the theory of self-adjoint extensions. 
One can show that the dimension of $\ker(S^*-\lambda_+)$ does not depend on $\lambda_+\in\dC^+$
and the dimension of $\ker(S^*-\lambda_-)$ does not depend on $\lambda_-\in\dC^-$. However, for fixed $\lambda_+\in\dC^+$ and $\lambda_-\in\dC^-$ 
and hence, in particular, for $\lambda_+=i$ and 
$\lambda_-=-i$,
the dimensions of $\ker(S^*-\lambda_+)$ and $\ker(S^*-\lambda_-)$ may be different. According to the second von Neumann formula both dimensions coincide
if and only if $S$ admits self-adjoint extensions in $\cH$.

\begin{theorem}\label{n2}
	Let $S$ be a densely defined closed symmetric operator in $\cH$. Then there exist self-adjoint extensions $A$ of $S$ in $\cH$ if and only if 
	\begin{equation*}
	\dim\bigl(\ker(S^*-i)\bigr)=\dim\bigl(\ker(S^*+i)\bigr).
	\end{equation*}
	If, in this case $U:\ker(S^*-i)\rightarrow \ker(S^*+i)$ is a unitary operator and $\dom(S^*)$ is decomposed as in \eqref{neu1}, then the operator $A$ defined by
	\begin{equation}\label{aa}
	\begin{split}
	A(f_S+f_i+f_{-i})&=Sf_S +if_i -if_{-i},\\
	\dom (A)&=\bigl\{f=f_S+f_i+f_{-i}\in\dom (S^*):f_{-i}=Uf_i \bigr\},
	\end{split}
	\end{equation}
	is a self-adjoint extension of $S$ and, vice versa, for any self-adjoint extension $A$ of $S$ there exists a unitary operator 
	$U:\ker(S^*-i)\rightarrow \ker(S^*+i)$ such that \eqref{aa} holds.
\end{theorem}

From the intuition is it clear that for a densely defined symmetric operator with equal infinite deficiency indices there is a lot of flexibility
for the unitary operators in between the defect subspaces (since they are infinite dimensional). This flexibility also allows to construct self-adjoint extensions
with various different spectral properties.

Now we wish to consider the following particular situation. Let again $S$ be a densely defined closed symmetric operator in $\cH$ with equal infinite
deficiency indices and assume that there exists a self-adjoint extension of $S$ such that the resolvent is a compact operator. In this situation 
we shall construct another self-adjoint extension $A$ of $S$ with prescribed 
essential spectrum in the next theorem.

\begin{theorem}\label{essit}
	Let $S$ be a densely defined closed symmetric operator in $\cH$ with equal infinite deficiency indices and assume that	
	there exists a self-adjoint extension of $S$ with compact resolvent.
	Let $\cG$ be a separable infinite dimensional Hilbert space and let $\Xi$ be a self-adjoint operator in $\cG$ with $\dR\cap\rho(\Xi)\not=\emptyset$. 
	Then there exists a self-adjoint extension
	$A$ of $S$ in $\cH$ such that
	\begin{equation}\label{essi}
	\sigma_{\rm ess}(A)=\sigma_{\rm ess}(\Xi).
	\end{equation}
\end{theorem}

\begin{proof}
Let $A_0$ be a self-adjoint extension of $S$ in $\cH$ such that the resolvent $(A_0-\lambda)^{-1}$ is a compact operator for some, and hence for all, $\lambda\in\rho(A_0)$.
 	Let us fix some point $\mu\in\dR\cap\rho(A_0)\cap\rho(\Xi)$. Note that this is possible since we have assumed $\dR\cap\rho(\Xi)\not=\emptyset$ and 
 	the spectrum of $A_0$ is a discrete subset of the real line due to the compactness assumption.
	In the present situation the spaces $\ker(S^*-\lambda_+)$ and $\ker(S^*-\lambda_-)$ for $\lambda_\pm\in\dC^\pm$
	are both infinite dimensional and one can show that here also the space $\ker(S^*-\mu)$ is infinite dimensional; this follows, e.g., from the direct
	sum decomposition
	\begin{equation}\label{deco}
	\dom (S^*)=\dom (A_0)\,\dot+\,\ker(S^*-\mu)
	\end{equation}
	and the fact that $S^*$ is an infinite dimensional extension of $A_0$. Moreover, it is no restriction to assume that the Hilbert space $\cG$ in the assumptions
	of the theorem coincides with $\ker(S^*-\mu)$ since any two separable infinite dimensional Hilbert spaces can be identified via a unitary operator.
	Now observe the orthogonal sum decomposition
	\begin{equation*}
	\cH=\ker(S^*-\mu)\oplus\ran(S-\mu)
	\end{equation*}
	and with respect to this decomposition we consider the bounded everywhere defined operator
	\begin{equation}\label{rmu}
	R_\mu:=(A_0-\mu)^{-1}+\left[\begin{matrix} (\Xi-\mu)^{-1} & 0 \\ 0 & 0 \end{matrix}\right].
	\end{equation}
	We claim that $R_\mu^{-1}$ is a well-defined operator. In fact, if $R_\mu h=0$ for some $h\in\cH$ then \eqref{rmu} implies
	\begin{equation*}
	(A_0-\mu)^{-1}h=-\left[\begin{matrix} (\Xi-\mu)^{-1} & 0 \\ 0 & 0 \end{matrix}\right]h,
	\end{equation*}
	and since the left-hand side belongs to $\dom (A_0)$ and the right-hand side is nonzero only in $\ker(S^*-\mu)$ it follows from the direct sum decomposition
	\eqref{deco} that $h=0$. This confirms that $R_\mu^{-1}$, and hence also 
	$$
	A:=R_\mu^{-1}+\mu
	$$
	is a well-defined operator. It is clear from $R_\mu=(A-\mu)^{-1}$ that $\mu\in\rho(A)$ and $A$ is self-adjoint in $\cH$ since the same is obviously true for $R_\mu$ in \eqref{rmu}. 
	In order to determine the essential spectrum of $A$ recall the Weyl theorem (see, e.g., \cite[Theorem XIII.14]{RS78}) which states 
	that compact perturbations in resolvent 
	sense do not change the essential spectrum. In the present situation we have that 
	$$
	(A-\mu)^{-1}-\left[\begin{matrix} (\Xi-\mu)^{-1} & 0 \\ 0 & 0 \end{matrix}\right]=R_\mu-\left[\begin{matrix} (\Xi-\mu)^{-1} & 0 \\ 0 & 0 \end{matrix}\right]=(A_0-\mu)^{-1}
	$$
	is a compact operator and hence the essential spectrum $\sigma_{\rm ess}((A-\mu)^{-1})$ coincides with the 
	essential spectrum of the diagonal block matrix operator, that is, 
	$$\sigma_{\rm ess}\left(\left[\begin{matrix} (\Xi-\mu)^{-1} & 0 \\ 0 & 0 \end{matrix}\right]\right)=\sigma_{\rm ess}((\Xi-\mu)^{-1})\cup\{0\}.$$ 
	This implies \eqref{essi}.  
\end{proof}

From the construction of the operator $A$ in the proof of Theorem~\ref{essit} the following representation can be concluded:
\begin{equation}\label{as}
\begin{split}
A(f_0+f_\mu)&=A_0f_0 + \mu f_\mu,\\
\dom (A)&=\bigl\{f_0+f_\mu\in\dom (A_0)\,\dot+\,\ker(S^*-\mu):\\
&\qquad\qquad\qquad\qquad\qquad (\Xi-\mu)f_\mu=P_\mu(A_0-\mu)f_0\bigr\};
\end{split}
\end{equation}
here $P_\mu$ denotes the orthogonal projection in $\cH$ onto $\ker(S^*-\mu)$.
In fact, since \eqref{rmu} is the resolvent $(A-\mu)^{-1}$ of $A$ it follows that the elements $f\in\dom (A)$ have the form 
$f=R_\mu h$, $h\in\cH$. Due to the direct sum decomposition \eqref{deco} we have $R_\mu h=f=f_0+f_\mu$ with some $f_0\in\dom (A_0)$ and some $f_\mu\in\ker(S^*-\mu)$,
and when comparing with \eqref{rmu} it follows that $f_0=(A_0-\mu)^{-1}h$ and $f_\mu=(\Xi-\mu)^{-1}P_\mu h$. Hence it is clear that $f=R_\mu h\in\dom (A)$
satisfies the condition
\begin{equation}\label{bcd}
(\Xi-\mu)f_\mu=P_\mu(A_0-\mu)f_0
\end{equation}
in \eqref{as}. On the other hand, if $f=f_0+f_\mu\in\dom (A_0)\,\dot+\,\ker(S^*-\mu)$ is such that \eqref{bcd} holds then one can verify in a similar way that
there exists $h\in\cH$ such that $f=R_\mu h$, and hence $f\in\dom (A)$. Summing up we have shown the representation \eqref{as}.

Finally we note that the explicit form \eqref{as} of $A$ comes via a restriction of the adjoint operator $S^*$ and the decomposition \eqref{deco}; the domain of $A$
is described by an abstract boundary condition depending on the choice of the operator $\Xi$. This abstract result can of course be 
formulated in various explicit situations, e.g., for infinitely many $\delta'$-interactions as in Section~\ref{sec2} or for the Laplacian
on a bounded domain as in Section~\ref{sec1}, where the boundary condition in \eqref{as} can be specified further. 

Furthermore, the self-adjoint extensions $A$ and $A_0$ can be described in the formalism of von Neumann's second formula in Theorem~\ref{n2}.
If one fixes a unitary operator $U_0:\ker(S^*-i)\rightarrow \ker(S^*+i)$ for the representation of $A_0$ in \eqref{aa} then 
the unitary operator $U:\ker(S^*-i)\rightarrow \ker(S^*+i)$ corresponding to the self-adjoint extension $A$ can be expressed in terms of $U_0$ and the parameter
$\Xi$. The technical details are left to the reader.

\numberwithin{theorem}{section}
\appendix

\section{Continuous dependence of the eigenvalues on varying domains}\label{appa}

In this appendix we establish an auxiliary result on the continuous dependence of the eigenvalues 
of the Neumann Laplacian on varying domains, which is useful and convenient for the proofs of Theorem~\ref{th-HSS} and Theorem \ref{thCdV+}.

For our purposes it is sufficient to consider the following geometric setting: 
Let $\Omega\subset\R^2$ be a bounded Lipschitz domain 
and assume that also the subdomains 
$$\Omega_\pm=\Omega\cap\left\{(x,y)\in\R^2: \pm x> 0\right\}$$
are (bounded, nontrivial) Lipschitz domains. Furthermore, we assume that the set 
$$\Gamma =\Omega\cap\left\{(x,y)\in\R^2: x=0\right\}=\partial\Omega^-\cap\partial\Omega^+$$ is a (compact) interval with the endpoints 
$(0,A)$ and $(0,B)$ in $\R^2$, where $A<B$. 
For $a,b\in [A,B]$ fixed such that $a\le b$ we introduce the domain $\Omega_{a,b}$ by 
\begin{equation}\label{omab}
\Omega_{a,b}=\Omega_-\cup\Omega_+ \cup \Gamma_{a,b},
\text{ where }\Gamma_{a,b}=\{0\}\times (a,b)
\end{equation}
(see Figure~\ref{fig5}, left). Note that
$\Omega_{A,B}=\Omega$ and $\Omega_{a,a}=\Omega_-\cup\Omega_+=\Omega\setminus\Gamma$. 
We denote by $Y\subset [A,B]\times [A,B]$ the set of all admissible pairs $\{a,b\}$, that is,
$$Y=\bigl\{\{a,b\} : A\leq a\leq b\leq B\bigr\}.$$

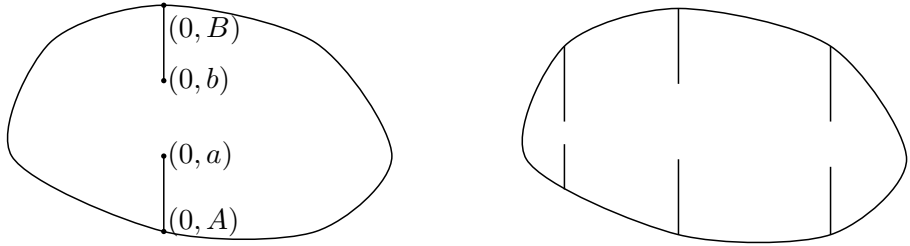
\begin{figure}[h]
\begin{tikzpicture}
\draw [line width=0.2mm, black] plot [smooth cycle] coordinates {(3,0) (2,1.5) (0,2) (-1.5,1.5) (-2,0) (0,-1) (2,-1)};

\draw [line width=0.2mm, black]  (0,2) -- (0,1) ;
\draw [line width=0.2mm, black]  (0,0) -- (0,-1);

\node[circle,fill=black,inner sep=0pt,minimum size=0.7mm ] (a) at (0,-1) {};
\node[circle,fill=black,inner sep=0pt,minimum size=0.7mm ] (a) at (0,0) {};
\node[circle,fill=black,inner sep=0pt,minimum size=0.7mm ] (a) at (0,1) {};
\node[circle,fill=black,inner sep=0pt,minimum size=0.7mm ] (a) at (0,2) {};

\node[text width=10mm] at (0.55,0) {$(0,a)$};
\node[text width=10mm] at (0.55,1) {$(0,b)$};
\node[text width=10mm] at (0.55,-0.87) {$(0,A)$};
\node[text width=10mm] at (0.55,1.68) {$(0,B)$};

\end{tikzpicture}\qquad\qquad
\begin{tikzpicture}
\draw [line width=0.2mm, black] plot [smooth cycle] coordinates {(3,0) (2,1.5) (0,2) (-1.5,1.5) (-2,0) (0,-1) (2,-1)};

\draw [line width=0.2mm, black]  (0,2) -- (0,1) ;
\draw [line width=0.2mm, black]  (0,0) -- (0,-1);

\draw [line width=0.2mm, black]  (-1.5,1.5) -- (-1.5,0.5) ;
\draw [line width=0.2mm, black]  (-1.5,0.2) -- (-1.5,-0.4);

\draw [line width=0.2mm, black]  (2,1.5) -- (2,0.5) ;
\draw [line width=0.2mm, black]  (2,-0.1) -- (2,-1);

\end{tikzpicture}
\caption{Domain $\Omega_{a,b}$ with one wall (left) and $m=3$ walls (right)}\label{fig5}
\end{figure}

Since the domain $\Omega_{a,b}$ in \eqref{omab} has the \textit{cone property} (see, e.g. \cite[Chapter~IV,~4.3]{A75}) it follows from
Rellich's  theorem \cite[Theorem~6.2]{A75} that the embedding $\H^1(\Omega_{a,b})\hookrightarrow \L(\Omega_{a,b})$ is compact. Therefore, 
the spectrum of the Neumann Laplacian $A_{\Omega_{a,b}}$ on $\Omega_{a,b}$ is purely discrete.
We denote by $\left(\lambda_k(\Omega_{a,b})\right)_{k\in\N}$  the sequence of eigenvalues of 
$A_{\Omega_{a,b}}$  numbered in nondecreasing order with multiplicities taken into account.

\begin{theorem}\label{th-contin}
For each $k\in\N$ the function
$\{a,b\}\mapsto \lambda_k(\Omega_{a,b})$
is continuous on $Y$.
\end{theorem}

\begin{remark}\label{rem-contin}
	Theorem~\ref{th-contin} remains valid for more general domains $\Omega_{a,b}$ obtained
	from $\Omega$ by adding $m>1$ walls in the same way -- see  Figure~\ref{fig5} (right, here $m=3$).
	In this case 
	$a=\{a_1,\dots,b_m\}$, $b=\{b_1,\dots,b_m\}$ with
	\begin{gather}\label{aabb}
	A_j\le a_j\leq b_j\leq B_j,\ j=1,\dots,m
	\end{gather}
	and $\{a,b\}\mapsto  \lambda_k(\Omega_{a,b})$ is continuous on 
	$\{\{a,b\}\in\R^{2m}: \eqref{aabb}\text{ holds}\}$.
\end{remark}

For the proof of Theorem~\ref{th-contin} we shall first recall a particular case of a more general abstract result established in \cite{IOS89}, 
which is formulated and proved for operators in \textit{varying} Hilbert spaces.

\begin{theorem}[Iosif'yan-Oleinik-Shamaev, 1989]\label{thIOS}
Let $B_n$, $n\in\N$, and $B$ be nonnegative compact operators in a Hilbert space $\HS$.
We denote by $(\mu_k(B_n))_{k\in\N}$ and $(\mu_k(B))_{k\in\N}$ the sequences of the eigenvalues of 
$B_n$ and $B$, respectively, numbered in nonincreasing order with multiplicities taken into account.
Assume that the following conditions hold:
\begin{itemize}
\item[{\rm (i)}] $\sup_{n}\|B_n\|<\infty$;

\item[{\rm (ii)}] $\forall f\in\HS$: $B_n f\to Bf$ as $n\to \infty$;

\item[{\rm (iii)}] for any bounded sequence $(f_n)_{n\in\N}$ in $\HS$ there exists $u\in\HS$ 
and a subsequence $(n_k)_{k\in\N}$ such that
$B_{n_k}f_{n_k}\to u$ in $\HS$ as $k\to\infty$.

\end{itemize}
Then for each $k\in\N$
\begin{gather}\label{mumu}
\mu_k(B_n)\to \mu_k(B)\text{ as }n\to\infty.
\end{gather}
\end{theorem}

\begin{proof}[Proof of Theorem~\ref{th-contin}] 
Fix some $\{a,b\}\in Y$ and consider an arbitrary  sequence $\{a_n,b_n\}\in Y$, $n\in\N$, such that 
$\lim_{n\to\infty}a_n=a$ and $\lim_{n\to\infty}b_n=b$.
We have to show that for each $k\in\N$
\begin{gather}
\label{lambdalambda}
\lambda_k(\Omega_{a_n,b_n})\to \lambda_k(\Omega_{a,b})\text{ as }n\to\infty.
\end{gather}
The strategy is  to apply Theorem~\ref{thIOS} to the resolvents of the Neumann Laplacians 
$A_{\Omega_{a_n,b_n}}$ and $A_{\Omega_{a,b}}$. 
More precisely, 
we consider the operators
\begin{gather*}
B_n=(A_{\Omega_{a_n,b_n}}+\Id)^{-1}\quad\text{and}\quad B=(A_{\Omega_{a,b}}+\Id)^{-1},
\end{gather*}
which are bounded operators acting in $\HS=\L(\Omega)=\L(\Omega_{a,b})$.
We show below that these operators satisfy the conditions (i)-(iii) in Theorem~\ref{thIOS}. 
Then it follows that \eqref{mumu} holds for each $k\in\N$ and from 
\begin{gather*}
\mu_k(B_n)=(\lambda_k(\Omega_{a_n,b_n})+1)^{-1}\quad\text{and}\quad 
\mu_k(B)=(\lambda_k(\Omega_{a,b})+1)^{-1}
\end{gather*}
we conclude \eqref{lambdalambda}.\\

(i) This condition holds since $$\|B_n\|={1\over \mathrm{dist}(-1,\,\sigma(A_{\Omega_{a_n,b_n}}))}=1.$$

(ii) In order to check condition (ii) in Theorem~\ref{thIOS} let $f\in \L(\Omega)$ and set 
$u_n=B_n f$, $n\in\N$. For $\phi\in \H^1(\Omega_{a_n,b_n})$
it follows from the definition of the Neumann Laplacian $A_{\Omega_{a_n,b_n}}$ that $u_n$ satisfies
\begin{equation}\label{weak_n}
(\nabla u_n,\nabla\phi)_{\L(\Omega_{a_n,b_n})}+
(u_n,\phi)_{\L(\Omega_{a_n,b_n})}= \bigl((A_{\Omega_{a_n,b_n}}+\Id)u_n,\phi\bigr)_{\L(\Omega_{a_n,b_n})}
=(f,\phi)_{\L(\Omega_{a_n,b_n})}.
\end{equation}
In particular, using \eqref{weak_n} for $\phi=u_n$ we get
\begin{equation*}
\begin{split}
\|u_n\|^2_{\H^1(\Omega_{a_n,b_n})}&=
\|\nabla u_n\|^2_{\L(\Omega_{a_n,b_n})}+\|u_n\|^2_{\L(\Omega_{a_n,b_n})}=(f,u_n)_{\L(\Omega_{a_n,b_n})}\\&\leq \|f \|_{\L(\Omega_{a_n,b_n})}\|u_n\|_{\L (\Omega_{a_n,b_n})}
  \leq\|f \|_{\L(\Omega_{a_n,b_n})}\|u_n\|_{\H^1(\Omega_{a_n,b_n})},
\end{split}
\end{equation*}
and therefore
\begin{gather}\label{apriori}
\|u_n\|_{\H^1(\Omega_{a_n,b_n})}\leq \|f \|_{\L(\Omega_{a_n,b_n})}.
\end{gather}

We set $u_n^\pm = u_n\!\!\restriction_{\Omega_\pm}$. Below  we shall use the same $\pm$-superscript notation for restrictions of other functions
onto $\Omega_\pm$. It follows from \eqref{apriori}
that  $(u_n^\pm)_{n\in\N}$ is a bounded sequence in 
$\H^1(\Omega_\pm)$ and hence there exist $u^\pm\in\H^1(\Omega_\pm)$
and a subsequence $n_k\to\infty$ such that
\begin{gather}
\label{H1}
u_{n_k}^\pm\rightharpoonup u^\pm\text{ in }\H^1(\Omega_\pm)
\end{gather}
(as usual the notation $\rightharpoonup$ is used for the \textit{weak} convergence).
With the help of Rellich's  theorem we conclude from \eqref{H1} that
\begin{gather}
\label{H1-}
u_{n_k}^\pm\to u^\pm\text{ in }\H^{1-\kappa}(\Omega_\pm),\quad \kappa\in (0,1].
\end{gather}
Finally, well-known mapping properties of the trace operator on $\H^1(\Omega_\pm)$ (see, e.g., \cite[Theorem~3.37]{McL00}) together with \eqref{H1-} lead to
\begin{gather}
\label{traces}
\gamma_{\Gamma}^\pm u_{n_k}^\pm \to \gamma_{\Gamma}^\pm u^\pm\text{ in }\L(\Gamma)
\end{gather}
as $n_k\to\infty$,
where $\gamma^\pm_{\Gamma}u^\pm$ stands for the restriction of the trace of the function $u^\pm\in\H^1(\Omega_\pm)$ onto $\Gamma$.

Next we introduce the set of functions 
\begin{multline*}
\widehat\H^1(\Omega_{a,b})=\bigl\{u\in\H^1(\Omega_{a,b}):\ \exists\delta=\delta(u)>0\text{ such that}\\ 
u=0\text{ in  } \delta\text{-neighborhoods of }(0,a)\text{ and }(0,b)\bigr\}.
\end{multline*}
It is known that $\widehat\H^1(\Omega_{a,b})$ is dense in $\H^1(\Omega_{a,b})$ 
(this is due to the fact that the capacity of the set $\{(0,a),\,(0,b)\}$ is zero; we refer to \cite{RT75} for more details). 
Now let $\phi\in \widehat\H^1(\Omega_{a,b})$. It is clear that for $n_k$ sufficiently large we also have 
$\phi\in \H^1(\Omega_{a_{n_k},b_{n_k}})$ and hence
\eqref{weak_n} is valid. The identity  \eqref{weak_n} written componentwise reads as 
\begin{equation*}
\begin{split}
(\nabla u_{n_k}^- ,\nabla\phi^-)_{\L(\Omega_-)}+
(\nabla u_{n_k}^+ ,\nabla&\phi^+)_{\L(\Omega_+)}+
(u_{n_k}^-,\phi^-)_{\L(\Omega_-)}+
(u_{n_k}^+,\phi^+)_{\L(\Omega_+)}\\
&=
(f^-,\phi^-)_{\L(\Omega_-)}+
(f^+,\phi^+)_{\L(\Omega_+)},
\end{split}
\end{equation*}
and passing to the limit (we have weak convergence in $\H^1(\Omega_\pm)$ by \eqref{H1}) as  $n_k\to\infty$
 we get 
\begin{equation}\label{weak}
\begin{split}
(\nabla u^- ,\nabla\phi^-)_{\L(\Omega_-)}+
(\nabla u^+ ,\nabla&\phi^+)_{\L(\Omega_+)}+
(u^-,\phi^-)_{\L(\Omega_-)}+
(u^+,\phi^+)_{\L(\Omega_+)}\\
&=
(f^-,\phi^-)_{\L(\Omega_-)}+
(f^+,\phi^+)_{\L(\Omega_+)}.
\end{split}
\end{equation}
Let us denote
$$
u(x)=
\begin{cases}
u^-(x),& x\in\Omega_-,
\\
u^+(x),& x\in\Omega_+.
\end{cases}
$$
Obviously $u\in\L(\Omega)$. Using \eqref{H1-} with $\kappa=1$ we obtain
\begin{gather}\label{ii}
u_{n_k}\to u\text{ in }\L(\Omega).
\end{gather}
Since $u_{n_k}\in\H^1(\Omega_{a_{n_k},b_{n_k}})$ it is clear that
$$\gamma^-_{\Gamma_{a_{n_k},b_{n_k}}} u^-_{n_k}=\gamma_{\Gamma_{a_{n_k},b_{n_k}}}^+ u_{n_k}^+,$$ 
where $\gamma^\pm_{\Gamma_{a_{n_k},b_{n_k}}}$ is the restriction of the trace onto $\Gamma_{a_{n_k},b_{n_k}}=\{0\}\times (a_{n_k},b_{n_k})$.
Therefore, \eqref{traces} implies that $\gamma^-_{\Gamma_{a',b'}} u^-=\gamma^+_{\Gamma_{a',b'}} u^+$ for \textit{any} interval $(a',b')\subset (a,b)$ and, consequently,   
$$\gamma^-_{\Gamma_{a ,b }} u^-=\gamma^+_{\Gamma_{a ,b }} u^+.$$ 
As $u^\pm\in \H^1(\Omega_\pm)$
this implies
$u\in \H^1(\Omega_{a,b})$ and \eqref{weak} can be written in the form 
\begin{gather}
\label{weak+}
(\nabla u ,\nabla\phi)_{\L(\Omega_{a,b})}+
(u,\phi)_{\L(\Omega_{a,b})}=
(f,\phi)_{\L(\Omega_{a,b})}.
\end{gather}
Since $\widehat\H^1(\Omega_{a,b})$ is dense in $ \H^1(\Omega_{a,b})$ this equality holds
for any $\phi\in \H^1(\Omega_{a,b})$. It is easy to see that 
\eqref{weak+} is equivalent to $u=Bf$. This also shows that the limit function $u$ is independent of the subsequence $n_k$ and hence we conclude that
\eqref{ii} holds for any subsequence $n_k$. Thus,
$$B_nf = u_n\to u = Bf\text{ in }\L(\Omega)$$
as $n\to\infty$. We have verified condition (ii) in Theorem~\ref{thIOS}.

(iii) To check this condition let $(f_n)_{n\in\N}$ be a bounded sequence in $\L(\Omega)$. 
The same arguments as in the proof of (ii) (cf.~\eqref{apriori}) show that the sequence 
$(B_n f_n)_{n\in\N}$ is bounded in $\H^1(\Omega\setminus\Gamma)$, and hence contains a weakly convergent subsequence in $\H^1(\Omega\setminus\Gamma)$.
Since the embedding 
$$H^1(\Omega\setminus\Gamma)\hookrightarrow  \L(\Omega\setminus\Gamma)=\L(\Omega)$$
is compact (again we use Rellich's embedding theorem) we conclude that there is a strongly convergent subsequence in $\L(\Omega)$, that is, 
condition (iii) in Theorem~\ref{thIOS} is satisfied.  
\end{proof}

\begin{remark}
	Besides the continuity of the function $\{a,b\}\mapsto \lambda_k(\Omega_{a,b})$
	one can also conclude that it decreases (\textit{resp.}, increases) monotonically  with respect to $a$ (\textit{resp.},  with respect to $b$). This follows easily from the min-max principle  (see, e.g., \cite[Section~4.5]{D95}).
	Note, that, in general, when one perturbs a fixed domain $\Omega$ by removing a subset $S_a$  ($a\in\R$ is a parameter)
	the monotonicity of the eigenvalues of the Neumann Laplacian in $\Omega_a:=\Omega\setminus S_a$ does not follow from the 
	the  monotonicity of the underlying domains with respect to $a$, i.e., even if $\Omega_a\subset \Omega_{\widetilde a}$,
	it does not mean that $\lambda(\Omega_a)\geq \lambda(\Omega_{\widetilde a})$ (see \cite[Section~2.3]{MNP85} for more details). 
	This is in contrast to Dirichlet Laplacian, where the monotonicity is always present -- see, e.g., \cite{RT75,GZ94,Oz81,MNP85} 
	for the properties of Dirichlet eigenvalues in so perturbed domains. However, in our configuration monotonicity nevertheless holds 
	for Neumann eigenvalues. 
	This is due to a special structure of the removed set having the form of two walls with zero thickness. 
\end{remark}

\section{Convergence results for monotone sequences of quadratic forms}\label{appb}

We recall a well-known convergence result for a sequence of
monotonically increasing quadratic forms from \cite{Si78} which is used in the proof of Theorem~\ref{th-BK}.

Consider a family $\{\a_{q}\}_{q>0}$ of densely defined closed nonnegative sesquilinear forms in a Hilbert space $\HS$. For simplicity we assume that 
the  domain of  $\a_q $ is the same for all $q$, and we use the notation $\dom(\a_q)=\HS^1$. 
Let $\mathcal{A}_q$ be the nonnegative self-adjoint operator associated with the form $\a_q$ via the first representation theorem.
Now assume, in addition, that the family $\{\a_{q}\}_{q>0}$ of forms increases monotonically as $q$ decreases, i.e.
for any $0<q<\widetilde q<\infty$ one has
\begin{gather}
\label{monot-forms}
\a_{q}[u,u]\geq \a_{\widetilde q}[u,u],\quad u\in \HS^1.
\end{gather}
We define the limit form $\a_0$ as follows:
$$\dom(\a_0)=\left\{u\in\HS^1:\ \sup\limits_{q>0}\a_{q}[u,u]<\infty\right\},\ \a_0[u,v]=\lim\limits_{q\to 0}\a_{q}[u,v].$$
One verifies that $\a_0$ is a well-defined nonnegative symmetric sesquilinear form (which is not necessarily densely defined) 
and, in fact, by \cite{Si78} the limit form $\a_0$ is closed.
Let us now assume that $\dom(\a_0)$ is dense in $\HS$, so that one can associate a 
nonnegative self-adjoint operator $\mathcal{A}_0$ with $\a_0$ via the first representation theorem.\footnote{We wish to mention here that in the situation where
the limit form $\a_0$ is not densely defined one associates a self-adjoint relation (multivalued operator) via the corresponding generalized first representation theorem
for nondensely defined closed nonnegative forms; see \cite{BHSW10} for more details and related convergence results.}
According to \cite{Si78} one then has convergence of the corresponding nonnegative self-adjoint operators in the strong resolvent sense (see also \cite[Theorem 4.2]{BHSW10}):

\begin{theorem}[Simon, 1978]
For each $f\in\HS$ one has
\begin{gather}\label{strong} 
\|(\mathcal{A}_q+\Id)^{-1}f- (\mathcal{A}_0+\Id)^{-1}f\|\to 0\text{ as }q\to 0.
\end{gather}
\end{theorem}
 
Now let us assume, in addition, that the spectra of the self-adjoint operators $\mathcal{A}_q$ and $\mathcal{A}_0$ are purely discrete.
We write $(\lambda_k(\mathcal{A}_q))_{k\in\N}$ and $(\lambda_k(\mathcal{A}_0))_{k\in\N}$
for the eigenvalues of these operators counted with multiplicities and ordered
as nondecreasing sequences. In this case one can conclude the following spectral convergence:

\begin{theorem}\label{th-s+}
For each $k\in\N$ one has
\begin{gather}\label{ev-conv}
\lambda_k(\mathcal{A}_q)\to 
\lambda_k(\mathcal{A}_0)\text{ as }q\to 0.
\end{gather}
\end{theorem}
 
\begin{proof}
The discreteness of the spectra of $\mathcal{A}_q$ and $\mathcal{A}_0$ is equivalent to the compactness of the resolvents
$(\mathcal{A}_q+\Id)^{-1}$ and $(\mathcal{A}_0+\Id)^{-1}$. Moreover
\eqref{monot-forms} implies (cf.~\cite[Proposition~1.1]{Si78})
$$(\mathcal{A}_q+\Id)^{-1}\leq (\mathcal{A}_{\widetilde q}+\Id)^{-1}$$
provided $0<q<\widetilde q<\infty$. Then by \cite[Theorem VIII-3.5]{K66} the strong convergence in 
\eqref{strong} becomes even convergence in the operator norm, that is, 
\begin{gather}\label{norm:res:conv} 
\|(\mathcal{A}_q+\Id)^{-1}- (\mathcal{A}_0+\Id)^{-1}\|\to 0\text{ as }q\to 0.
\end{gather}
It is well-known (see, e.g., \cite[Corollary~A.15]{P06})  that the norm resolvent convergence
\eqref{norm:res:conv} implies the convergence
of the eigenvalues, i.e. \eqref{ev-conv} holds.
\end{proof}

\end{document}